\documentclass[leqno,11pt, a4]{amsart}
\usepackage[usenames,dvipsnames,svgnames,table]{xcolor}
\usepackage{tikz}
\usepackage[active]{srcltx}
\usepackage{pb-diagram}

\usepackage{dsfont}

\usepackage{mathrsfs}
\usepackage{amsmath}
\usepackage{amssymb}
\usepackage{amscd}
\usepackage{amsthm}
\usepackage[latin1]{inputenc}
\usepackage{graphics}
\usepackage{varioref}

\usepackage[latin1]{inputenc}
\usepackage{graphics}
\usepackage{pdfsync}

\labelformat{enumi}{(#1)}

\newtheorem{thm}[equation]{Theorem}
\newtheorem{defin}[equation]{Definition}
\newtheorem{remark}[equation]{Remark}
\newtheorem{prop}[equation]{Proposition}
\newtheorem{cor}[equation]{Corollary}
\newtheorem{lemma}[equation]{Lemma}
\newtheorem{ese}[equation]{Example}

\newtheoremstyle{named}{}{}{\itshape}{}{\bfseries}{.}{.5em}{\thmnote{#3}#1}
\theoremstyle{named}

\newcommand{\Keler}             {K\"{a}hler }

\newcommand{\Ga}{\Gamma}
\newcommand{\OO}{\mathcal{O}}
\newcommand{\meno}{^{-1}}

\newcommand{\PP}{\mathbb{P}}

\newcommand{\ext}{\operatorname{ext}} 

\newcommand{\liu}{\mathfrak{u}}
\newcommand{\liek}{\mathfrak{k}}

\newcommand{\lieg}{\mathfrak{g}}

\newcommand{\liep}{\mathfrak{p}}
\newcommand{\lieq}{\mathfrak{q}}

\newcommand{\liez}{\mathfrak{z}}
\newcommand{\lia}{\mathfrak{a}}
\newcommand{\lier}{\mathfrak{r}}
\newcommand{\supp}{\operatorname{supp}}
\newcommand{\spam}{\,\operatorname{span}\, }

\newcommand{\alfa}{\alpha}

\newcommand{\vacuo}{\emptyset}

\newcommand{\la}{\lambda}

\newcommand{\enf}{\emph}

\newcommand{\desudt}[1] []      {\dfrac {\mathrm {d} #1 }{\mathrm {dt}}}
\newcommand{\desudtzero}        {\desudt \bigg \vert _{t=0} }

\newcommand{\restr}[1]          {\vert_{#1}}

\newcommand{\Ad}{\operatorname{Ad}}
\newcommand{\ad}{{\operatorname{ad}}}
\newcommand{\sx}{\langle}
\newcommand{\xs}{\rangle}
\newcommand{\scalo}{\sx \cdot , \cdot \xs}
\newcommand{\relint}{\operatorname{relint}}

\newcommand{\Sl}{\operatorname{SL}}
\newcommand{\Gl}{\operatorname{GL}}

\newcommand{\End}{\operatorname{End}}
\newcommand{\lds}{\ldots}
\newcommand{\cds}{\cdots}
\newcommand{\cd}{\cdot}
\renewcommand{\setminus}{-}

\newcommand{\tr}{\operatorname{tr}}
\newcommand{\ra}{\rightarrow}
\newcommand{\lra}{\longrightarrow}
\newcommand{\C}{\mathbb{C}}
\newcommand{\R}{\mathds{R}}

\newcommand{\om}{\omega}

\renewcommand{\phi}{\varphi}

\renewcommand{\Bigl}{\left}

\renewcommand{\Bigr}{\right}
\newcommand{\Crit}{\operatorname{Crit}}

\newcommand{\CF}{\mathrm{C}_F}

\newcommand{\meo}{\end{document}}
\newcommand{\mup}{\mu_\liep}
\newcommand{\mua}{\mu_\lia}
\newcommand{\mupb}{\mu_\liep^\beta}

\newcommand{\roots}{\Delta} 

\newcommand{\liem}{\mathfrak{m}}
\newcommand{\simple}{\Pi} 
\newcommand{\lien}{\mathfrak{n}}
\newcommand{\metrica}{(\, , \, )}
\newcommand{\XS}{\overline{X}^S_\tau}
\newcommand{\Herm}{\mathcal{H}}
\newcommand{\Pos}{\mathcal{P}}
\begin{document}
\title{Satake-Furstenberg compactifications and gradient map}
\author{Leonardo Biliotti}
\address{(Leonardo Biliotti) Dipartimento di Scienze Matematiche, Fisiche e Informatiche \\
          Universit\`a di Parma (Italy)}
\email{leonardo.biliotti@unipr.it}
\begin{abstract}
Let $G$ be a real semisimple Lie group with finite center and let $\lieg=\liek \oplus \liep$ be a Cartan decomposition of its Lie algebra. Let $K$ be a maximal compact subgroup of $G$ with Lie algebra $\liek$ and let $\tau$ be an irreducible representation of $G$ on a complex vector space $V$. Let $h$ be a Hermitian scalar product on $V$ such that $\tau(G)$ is compatible with respect to $\mathrm{SU}(V,h)^\C$. We denote by $\mup:\mathbb P(V) \lra \liep$ the $G$-gradient map and by $\mathcal O$ the unique closed orbit of $G$ in $\mathbb P(V)$,
which is a $K$-orbit \cite{gjt,heinzner-schwarz-stoetzel}, contained in the unique closed orbit of the Zariski closure of $\tau(G)$ in $\mathrm{SU}(V,h)^\C$. We prove that up to equivalence the set of irreducible representations of parabolic subgroups of $G$ induced by $\tau$ are completely determined by the facial structure of the polar orbitope  $\mathcal E=\mathrm{conv}(\mup(\mathcal O))$.
Moreover, any parabolic subgroup of $G$ admits a unique closed orbit in $\mathcal O$ which is well-adapted to $\mup$.
These results are new also in the complex reductive case. The connection between $\mathcal E$ and $\tau$ provides a geometrical description of the Satake compactifications without root data. In this context the properties of the  Bourguignon-Li-Yau map are also investigated. Given a measure $\gamma$ on $\mathcal O$, we construct a map  $\Psi_\gamma$ from the Satake compactification of $G/K$ associated to $\tau$ and $\mathcal E$. If $\gamma$ is a  $K$-invariant measure then $\Psi_\gamma$ is an homeomorphism of the Satake compactification and $\mathcal E$. Finally, we  prove that for a large class of measures the map $\Psi_\gamma$ is surjective.
\end{abstract}

%
\thanks{The author was partially supported by the Project PRIN 2015, ``Real and Complex Manifolds: Geometry, Topology and Harmonic Analysis'', Project PRIN  2017 ``Real and Complex Manifolds: Topology, Geometry and holomorphic dynamics'' and by GNSAGA INdAM.}
\maketitle
\tableofcontents{}
\section{Introduction}
Let $G$ be a semisimple noncompact real Lie group with finite center and let $K$ be a maximal compact subgroup of $G$. Then $X=G/K$ is a symmetric space of noncompact type.  Let $\tau:G \lra \mathrm{SL}(V)$ be an irreducible complex representation. There exists a Hermitian scalar product $h$ on $V$ such that $\tau(G)\subset \mathrm{SU}(V,h)^\C$ is compatible with respect to the Cartan decomposition of $\mathrm{SU}(V,h)^\C=\mathrm{SU}(V,h)\exp (i\mathfrak{su}(V,h))$, where $\mathfrak{su}(V,h)=\mathrm{Lie}(\mathrm{SU}(V,h))$ \cite[4.32 Proposition]{gjt}. In this paper we investigate the projective representation $\tau:G\lra \mathrm{PSL}(V)$. We will write $gv$ instead of $\tau(g)v$ for simplicity.
There is a corresponding  $G$-gradient map $\mup:\mathbb P(V) \lra \liep$. This map is $K$-equivariant  and for any $\beta\in \liep$, the gradient of the smooth function $
\mup^\beta:\mathbb P(V) \lra \R,
$
$\mup^\beta(x)=\langle \mup(x),\beta \rangle$ where $\langle \cdot, \cdot \rangle$ is an $\mathrm{Ad}(K)$-invariant scalar product on $\liep$, is given by
the fundamental vector field $\beta^\#$ induced by the $G$- action on $\mathbb P(V)$, i.e., $\beta^\# (p):=\desudtzero \exp(t\beta) p$. If $\lia$ is
an Abelian subalgebra of $\lieg$ contained in $\liep$ and $\pi_\lia$ is the orthogonal projection of $\liep$ onto $\lia$,
then $\mua = \pi_\lia \circ \mup$ is the $A =\exp(\lia)$-gradient map. The gradient map has been extensively studied in \cite{heinz-stoezel,heinzner-schwarz-Cartan,heinzner-schwarz-stoetzel} where the authors develop a geometrical invariant theory for actions of real Lie groups on real submanifolds of complex spaces.

The Zariski closure of $\tau(G)$ in
$\mathrm{U}(V,h)^\C$  is a semisimple complex Lie group \cite[Proposition 3.4]{heinz-stoezel} that we denote by $U^\C$. It acts irreducibly on $V$. By Borel-Weil Theorem there exists a unique closed orbit of the $U^\C$-action on $\mathbb P(V)$, that we denote by $\OO'$.  By a Theorem of Wolf (Lemma \ref{closed-G-orbit})
there exists a unique $G$ closed orbit $\OO$ contained in $\OO'$. $\mathcal O$ is a $K$-orbit \cite{heinzner-schwarz-stoetzel} and it captures much of the informations of the projective representation $\tau$ and of the $G$-gradient map $\mup$.
\begin{prop}
Let $A=\exp(\lia)$, where $\lia \subset \liep$ is an Abelian subalgebra. Then
\[
\mua(\mathbb P(V))=\mua(\mathcal O).
\]
If $\lia$ is a maximal Abelian subalgebra, then $\mua(\mathbb P(V))$ is the convex hull of a Weyl-group orbit.
\end{prop}
As an application,  the weights of $\tau$ are contained in the convex hull of the Weyl group orbit of $\mu_\tau$, where $\mu_\tau$ is the highest weight with respect to a positive Weyl chamber (see Proposition \ref{weight}).

For $\beta \in \liep$, let $G^{\beta+}=\{g\in G:\, \lim_{t\mapsto -\infty} \exp(t\beta)g\exp(-t\beta)\, \mathrm{exists}\, \}$. Then $G^{\beta+}$ is parabolic with Levi factor $G^{\beta}=\{g\in G:\, \mathrm{Ad}(g)(\beta)=\beta\}$.  It is well-known that any parabolic subgroup of $G$ arises as a $G^{\beta+}$ for some $\beta \in \liep$  \cite{biliotti-ghigi-heinzner-2,kirwan}. In Section \ref{representation-gradient map}, we prove (see Theorem \ref{parabloic-representation}) the following result.
\begin{thm}\label{teo-parabolici}
Let $\beta \in \liep$ and let $W$ the eigenspace associated to the maximum eigenvalue of $\beta$. Then:
\begin{enumerate}
\item $\mathbb P(W)=\{z \in \mathbb P(V):\, \mup^\beta(z)=\mathrm{max}_{y\in \mathbb P(V)}\, \mup^\beta\}$;
\item $G^{\beta+}$ preserves $W$ and acts irreducibly on $W$;
\item $W$ is the unique subspace of $V$ satisfying $b)$;
\item $G^{\beta+}$ has a unique closed orbit in $\mathcal O$ given by
\[
\mathbb P(W) \cap \mathcal O =\{z\in \mathcal O:\, \mup^\beta (z)=\mathrm{max}_{y\in \mathcal O}\, \mup^\beta \}.
\]
This orbit is connected, it is a $(K^\beta)^o$-orbit, and it is full in $\mathbb P(W)$.
\end{enumerate}
\end{thm}
Hence the unique closed orbit of $G^{\beta+}$ contained in $\mathcal O$ is well-adapted to $\tau$.
Observe that $\mup(\mathcal O)$ is a $K$-orbit but it is not true in general that $\mup$ defines a homeomorphism between $\mathcal O$ and $\mup(\mathcal O)$, as in the complex case,. Therefore, Theorem $1.2$ in \cite[pag. $582$]{biliotti-ghigi-heinzner-2} does not apply in our context.
The above result can be given in terms of the gradient flow of $\mup^\beta$ (see Theorem \ref{flow-parabolic}).

Let $\mup^\beta :\mathcal O \lra \R$ and let $W$ be the eigenspace associated to the maximum eigenvalue of $\beta$. $\mup^\beta$ is a Morse-Bott function and the unstable manifold relative to the maximum is the set  $W_1^\beta=\mathcal O \setminus \mathbb P (W^\perp )$. The gradient flow restricted to $W_1^\beta$ is given by
\[
\phi_\infty ([x])=\lim_{t\mapsto +\infty} \exp(t\beta)[x]=[\pi_W (x)],
\]
where $\pi_W:V\lra W$ is the orthogonal projection onto $W$. This means  $\phi_\infty (W_1^\beta)$ coincides with the unique closed orbit of $G^{\beta+}$ contained in $\mathcal O$. Now, we discuss about the relation between the parabolic subgroups of $G$ and the image of the gradient map.

The set $\mup(\mathcal O)$ is a $K$-orbit in $\liep$ and its convex hull, that we denote by $\mathcal E$, is a highly symmetric convex body in $\liep$. It  is called a \emph{polar orbitope} since the $K$-action on $\liep$ is polar \cite{dadok-polar,helgason}. The  convex hull of an orbit of an orthogonal representation is called an \emph{orbitope} \cite{orbitope}. These convex bodies have been largely studied in \cite{bgh-israel-p,gichev-polar,kobert-Scheiderer,orbitope} amongst many others.

Let  $P=\mathcal E\cap \lia$, where $\lia\subset\liep$ is a maximal Abelian subalgebra. By a theorem of Kostant $P=\pi_\lia(\mathcal E)=\mua(\mathcal O)$ and it is a polytope  \cite{kostant-convexity}. Hence any face of $P$ is exposed \cite{schneider-convex-bodies} and so any face of $\mathcal E$ is exposed as well.  The compact group $K$ acts on the set of the faces of $\mathcal E$ and the Weyl group $\mathcal W=N_K (\lia)/ K^{\lia}$ acts on the set of faces of $P$, where $N_K(\lia)=\{k\in K:\, \mathrm{Ad}(k)(\lia)=\lia\}$ is the normalizer of $\lia$ and $K^{\lia}=\{k\in K:\, \mathrm{Ad}(k)(z)=z,\, \mathrm{for\ any}\, z \in \lia\}$ is the centralizer of $\lia$ in $K$. Let  $\mathscr F(\mathcal E)$ and $\mathscr F(P)$ denote the sets of faces of $\mathcal E$ and $P$, respectively.
In \cite{biliotti-ghigi-heinzner-2} the authors proved $\mathscr F (P)/\mathcal W \cong \mathscr F (\mathcal E) /K$. This means that $P$ completely determines the boundary structure of $\mathcal E$ in the sense of convex geometry. In Section \ref{representation-gradient map},  we prove (see Theorem \ref{rappresentazioni-parabolici}) the following result.
\begin{thm}\label{pop}
There exists a bijection between $\mathscr F (P)/\mathcal W$ and the set of irreducible representations of parabolic subgroups of $G$ induced by $\tau$, up to equivalence.
\end{thm}
This correspondence between $\tau$ and the gradient map provides a geometrical description of the Satake compactification associated to $\tau$.

There are many different compactifications of symmetric spaces of noncompact type and they have different properties. For instance, Satake compactifications is used in the proof of the celebrated Mostow rigidity Theorem \cite{Mosr}. Other applications are given in harmonic analysis and global geometry \cite{borel-ji-libro,gjt} and in ergodic theory and probability theory in the context of sets of subgroups of $\mathrm{GL}(n,\R)$ \cite{fur,gm}.
The Satake compactifications of $X=G/K$ are obtained by embedding $X$ into some compact ambient space. These embeddings are given by  faithful projective representations of $G$. They are $G$-equivariant and so the $G$-action on $X$ extends to the compactifications.
We briefly recall the construction.

Let $\tau:G \lra \mathrm{PSL}(V)$ be a faithful irreducible projective representation. Then the map
\[
i_\tau: G/K \lra \mathbb P (\mathcal H (V)), \qquad gK \mapsto [gg^*],
\]
where $g^*$ denotes the adjoint of $g$ with respect to the Hermitian scalar product $h$ and $\mathcal H (V)$ denotes the set of Hermitian endomorphism of $V$, is well-defined and injective. $G$ acts on  $\mathbb P (\mathcal H (V))$ as follows $g[A]:=[gAg^*]$, so $i_\tau$ is $G$-equivariant. The closure $\XS=\overline{i_\tau(X)}$ in $\mathbb P (\mathcal H (V))$ is called the Satake compactification associated to $\tau$. Satake \cite{borel-ji-libro,Satake} gave a description of the boundary $\partial \XS:=\XS \setminus i_\tau (X)$ in terms of  $\mu_\tau$-connected subsets of simple roots \cite{Satake}. We replace $\mu_\tau$-connected subsets of simple roots with $\tau$-connected subspaces of $V$.  A $\tau$-connected subspace $W$ of $V$ is well-adapted to $\mathcal O$ and it is described in terms of the facial structure of $\mathcal E$. Indeed (see Section \ref{satake-new-description}), we prove that
$\mup(\mathbb P(W) \cap \mathcal O)=\mathrm{ext}\, F_W$, where $F_W$ is a face of $\mathcal E$ and $\mathrm{ext}\, F_W$ denotes the set of extreme points of $F_W$. We recall that $K$-action on a $K$-orbit in $\liep$ extends to a $G$-action \cite{heinzner-stoetzel-global}. Since $\mathrm{ext}\, F_W$ is contained in the $K$-orbit $\mup(\OO)$, we define
\[
Q(F_W):=\{a\in G:\, a\,\mathrm{ext}\, F_W =\mathrm{ext}\, F_W\}.
\]
We prove
\[
Q(W):=\{g\in G:\, gW=W\}=Q(F_W )
\]
and $Q(W)$ is a parabolic subgroup of $G$ acting irreducibly on $W$. If $\beta \in \CF^{H_F}$, i.e., $F_W=F_\beta (\mathcal E)$ is the exposed face of $\mathcal E$ defined by $\beta$, which is fixed by $H_F=K\cap Q(W)$ (see Subsection \ref{comvex-geometry}), then $Q(W)=G^{\beta+}$ and \[\mathbb P(W)\cap \mathcal O=\{z\in \mathcal O:\, \mup^\beta (z)=\mathrm{max}_{x\in \mathcal O}\, \mup^\beta \}\]
is the unique closed orbit of $Q(W)$ contained in $\mathcal O$.
This means that the information captured by $W$ only depends on the face $F_W$.

In Section \ref{Satake-sempre-Satake}, following the strategy of Biliotti and Ghigi \cite{biliotti-ghigi-American}, the  boundary components of Satake are described by $\tau$-connected subspaces (see Theorem \ref{Satake-mutauconnessi}). This allows the interpretation of $\XS$ in terms of rational self maps of $\mathcal O$ (see Lemma \ref{immagine-mappa-razionale} and Theorem \ref{Satake-razionale}). Roughly speaking, if $W$ is a $\tau$-connected subspace then the boundary component of $\XS$ corresponding to $W$ corresponds to rational maps $\mathcal O \dashrightarrow \mathcal O (Q(W))$, where $\mathcal O (Q(W))$ denotes the unique closed orbit of $Q(W)$ contained in $\mathcal O$. These maps are the composition of an automorphism of $\mathcal O (Q(W))$ and the gradient flow of $\mup^\beta$ restricted to the unstable manifold relative to the maximum, where $\beta \in \liep$ satisfies $F_W=F_\beta (\mathcal E)$. As a consequence, the Satake compactification associated to $\tau$ is completely described in terms of the facial structure of $\mathcal E$. This is new and the main difference between our paper and the papers \cite{biliotti-ghigi-American,korany,Satake}.
In Section \ref{srbly}, we apply these results to the Bourguignon-Li-Yau map.

Given a probability measure $\gamma$ on $\mathcal O$, we define
\[
\Psi_\gamma : G/K \lra \liep, \qquad gK \mapsto \int_{\mathcal O} \mup (\sqrt{gg^*} x) \mathrm d \gamma (x),
\]
which is called Bourguignon-Li-Yau map. The element $\rho(g)=\sqrt{gg^*}$ is the unique positive Hermitian endomorphism of $V$ such that $\rho(g)^{-1}g\in K$ and so $g=\rho(g)k$ is the polar decomposition of $g$.
This map has been studied at different levels of generality by Hersch \cite{hersch}, Millson and Zombro \cite{millson-zombro}, Bourguignon, Li and Yau \cite{bourguignon-li-yau} and Biliotti and Ghigi \cite{biliotti-ghigi-American},
to determine upper bounds for the first eigenvalue of the Laplacian acting on functions.
The image of the Bourguignon-Li-Yau lies in $\mathcal E$. We say that $\gamma$ is $\tau$-admissible if for any hyperplane $H$ of $\mathbb P(V)$, we have $\gamma(\mathcal O\cap H)=0$.  The interpretation of the elements of $\XS$ as rational maps gives a way to extend the Bourguignon-Li-Yau map to $\XS$. In Section \ref{srbly}, we prove the following result (Theorem \ref{bly-invariant-measure} and Theorem \ref{grado}).
\begin{thm}\label{stable-measure}
Let $\gamma$ be a $\tau$-admissible measure. Then $\Psi_\gamma$ extends to $\XS$ as a continuous map and $\Psi_\gamma (\partial \XS)\subseteq \partial \mathcal E$. Moreover, $\Psi_\gamma (\XS)=\mathcal E$ and $\Psi_\gamma (X)=\mathrm{Int}(\mathcal E)$.
\end{thm}
Since $0\in \mathrm{Int}\, \mathcal E$, see Lemma \ref{basic}, for any measure $\zeta$ defined by a Riemannian metric on $\mathcal O$, we can move the components of the gradient map with an automorphism of $\mathcal O$ in such a way these functions become functions of zero mean with respect to $\zeta$. Hence,  one may apply the Rayleigh Theorem to get an estimate of the first eigenvalue. We discuss this estimative in Section \ref{laplacian}. If $\nu$ is a $K$-invariant measure on $\OO$, see Theorem \ref{bly-invariant-measure}, we get the following result.
\begin{thm}
$\Psi_\nu$ defines an homeomorphism of $\XS$ onto $\mathcal E$.
\end{thm}
Satake compactification of a symmetric space of noncompact type is homeomorphic to a polar orbitope. This result is not new since Kor\'any \cite{korany} showed that $\XS$ is homeomorphic to $\mathcal E$. On the other hand the map used by Kor\'any is different from $\Psi_\nu$ and the techniques used by Kor\'any is different from ours. The above Theorems generalize results proved in \cite{biliotti-ghigi-American} for symmetric space of noncompact type of type $IV$. We point out that some parts of the paper \cite{biliotti-ghigi-American} are extremely technical (see Section $2.4$ p. $249-251$ and Section $2.5$) since when the paper \cite{biliotti-ghigi-American} was written the authors were not aware of the facial structure of invariant convex compact subsets of polar representations \cite{biliotti-ghigi-heinzner-1,biliotti-ghigi-heinzner-2,bgh-israel-p}. The results of this paper generalize and better clarify some results proved in \cite{biliotti-ghigi-American}.
Finally, we give a short proof of a Theorem of Moore \cite{moore} (see Theorem \ref{furst}), following the strategy of \cite{biliotti-ghigi-American}, stating that the Satake and Furstenberg compactifications are homeomorphic.

{\bfseries \noindent{Acknowledgements.}}   We wish to thank Alessandro Ghigi, Peter Heinzner and Lorenzo Nicolodi for interesting discussions. We would like to thank the anonymous referee for carefully reading our paper and for giving such constructive comments which substantially helped improving the quality of the paper.
\section{Preliminaries}
\subsection{Convex geometry}\label{comvex-geometry}
It is useful to recall a few definitions and results regarding convex
 sets. The reader may refer for instance  to \cite{schneider-convex-bodies} for more details.

 Let $V$ be a real vector
 space with a scalar product $\scalo$ and let $E\subset V$ be a
 compact convex subset.  The \emph{relative interior} of
 $E$, denoted $\mathrm{relint} E$, is the interior of $E$ in its affine hull.
 A face $F$ of $E$ is a convex subset $F\subset E$ with the following
 property: if $x,y\in E$ and $\mathrm{relint}[x,y]\cap F\neq \emptyset$, then
 $[x,y]\subset F$.  We say that a point $x\in E$ is an  \emph{extreme point}, and write $x\in \mathrm{ext}\, E$, if $\{x\}$ is a face. By a Theorem of Minkowski, $E$ is the convex hull of its extreme points \cite[p.19]{schneider-convex-bodies}. The faces of $E$
 are closed \cite[p. 62]{schneider-convex-bodies}.  A face distinct
 from $E$ and $\vacuo$ will be called a \enf{proper face}.  The
 \emph{support function} of $E$ is the function $ h_E : V \lra \R$, $
 h_E(u) = \max_{x \in E} \langle x, u \rangle$.  If $ u \neq 0$, the
 hyperplane $H(E, u) : = \{ x\in E : \langle x, u \rangle = h_E(u)\}$ is
 called the \emph{supporting hyperplane} of $E$ for $u$. The set
   \begin{gather}
     \label{def-exposed}
     F_u (E) : = E \cap H(E,u)
   \end{gather}
   is a face and it is called the \emph{exposed face} of $E$ defined by
   $u$.
In general not all faces of a convex subset are exposed.
  \begin{lemma}
[\protect{\cite[Lemma 3]{biliotti-ghigi-heinzner-1}}]
\label{ext-facce}
If $F$ is a face of a convex set $E$, then $\ext F = F \cap \ext E$.
 \end{lemma}
 \begin{lemma}[\protect{\cite[Lemma 8]{biliotti-ghigi-heinzner-1}}]
   \label{face-chain}
   If $E$ is a compact convex set and $F\subset E$ is a face, then
   there is a chain of faces $ F_0=F \subsetneq F_1 \subsetneq \cds
   \subsetneq F_k=E $ which is maximal, in the sense that for any $i$
   there is no face of $E$ strictly contained between $F_{i-1}$ and
   $F_i$.
 \end{lemma}
\begin{lemma}[\protect{\cite[Prop.5]{biliotti-ghigi-heinzner-1}}] \label{u-cono}
   If $F \subset E$ is an exposed face, the set $\CF : = \{ u\in V:
   F=F_u(E) \}$ is a convex cone. If $K$ is a compact subgroup of
   $O(V)$ that preserves both $E$ and $F$, then $\CF$ contains a fixed
   point of $K$.
 \end{lemma}
We denote by $\CF^K$ the elements of $\CF$ fixed by a compact group $K$.

The following result is well-known and a proof is given in \cite[p. 62]{schneider-convex-bodies}
\begin{thm} \label{schneider-facce} If $E$ is a compact convex set and
   $F_1,F_2$ are distinct faces of $E$, then $\relint F_1 \cap \relint
   F_2=\vacuo$. If $G$ is a nonempty convex subset of $ E$ which is
   open in its affine hull, then $G \subset\relint F$ for some face
   $F$ of $E$. Therefore $E$ is the disjoint union of the
     relative interiors of its faces.
 \end{thm}
The following result will be used to determine the image of the gradient map.
\begin{prop}\label{convex-criterium}
Let $C_1 \subseteq C_2$ be two compact convex subsets of $V$. Assume that for any $\beta \in V$ we have
\[
\mathrm{max}_{y\in C_1} \langle y , \beta \rangle=\mathrm{max}_{y\in C_2} \langle y , \beta \rangle.
\]
Then $C_1=C_2$.
\end{prop}
\begin{proof}
We may assume without loss of generality  that the affine hull of $C_2$ is $V$.
Assume by contradiction that $C_1 \subsetneq C_2$. Since $C_1$ and $C_2$ are both compact, it follows that there exists $p\in \partial C_1$ such that $p\in \stackrel{o}{C_2}$. Since every face of a compact convex set is contained in an exposed face \cite{schneider-convex-bodies}, there exists $\beta \in V$ such that
\[
\mathrm{max}_{y\in C_1} \langle y , \beta \rangle=\langle p, \beta \rangle.
\]
This means the linear function $x\mapsto \langle x, \beta \rangle$ restricted on $C_2$ achieves its maximum at an interior point which is a contradiction.
\end{proof}
\subsection{Compatible subgroups and parabolic subgroups}\label{compatible-parabolic}
In the sequel we always refer to \cite{borel-ji-libro,gjt,heinzner-schwarz-stoetzel}, see also \cite{borel-book,knapp-beyond}.

Let $U$ be compact connected Lie group. Let $U^\C$ be its universal complexification which
 is a linear reductive complex algebraic group \cite{akhiezer}. We
   denote by $\theta$ both the conjugation map $\theta : \liu^\C \lra
   \liu^\C$ and the corresponding group isomorphism $\theta : U^\C \lra
   U^\C$.  Let $f: U \times i\liu \lra U^\C$ be the diffeomorphism
 $f(g, \xi) = g \exp \xi$.  Let $G\subset U^\C$ be a closed
 subgroup. Set $K:=G\cap U$ and $\liep:= \lieg \cap i\liu$.  We say
 that $G$ is \emph{compatible} if $f (K \times \liep) = G$.  The
 restriction of $f$ to $K\times \liep$ is then a diffeomorphism onto
 $G$. Hence $\lieg=\liek\oplus \liep$ is the familiar Cartan decomposition  and so  $K$ is a maximal compact subgroup of $G$.  Note that $G$ has finitely many connected components. Since $U$ can be embedded in $\Gl(N,\C)$ for
 some $N$, and any such embedding induces a closed embedding of
 $U^\C$, any compatible subgroup is a closed linear group. Moreover
 $\lieg$ is a real reductive Lie algebra, hence $\lieg =
 \liez(\lieg)\oplus [\lieg, \lieg]$. Denote by $G_{ss}$ the analytic
 subgroup tangent to $[\lieg, \lieg]$. Then $G_{ss}$ is closed and
 $G^o=Z(G)^o \cd G_{ss}$ \cite[p. 442]{knapp-beyond}, where $G^o$, respectively $Z(G)^o$, denotes the connected component of  $G$, respectively of $Z(G)$, containing $e$.
The following lemma is well-known.
 \begin{lemma}$\, $ \label{lemcomp}
   \begin{enumerate}
   \item \label {lemcomp1} If $G\subset U^\C$ is a compatible
     subgroup, and $H\subset G$ is closed and $\theta$-invariant,
     then $H$ is compatible if and only if $H$ has only finitely many connected components.
   \item \label {lemcomp2} If $G\subset U^\C$ is a connected
     compatible subgroup, then $G_{ss}$ is compatible.
     \item \label{lemcomp3} If $G\subset U^\C$ is a compatible
       subgroup, and $E\subset \liep$ is any subset, then \[G^E=\{g\in G:\, \mathrm{Ad}(g)(z)=z,\, \forall z\in E\}\] is
       compatible. Indeed,  $G^{E}=K^{E} \exp(\liep^E)$,
where \[K^{E}=K\cap G^{E}=\{g\in K:\, \mathrm{Ad}(g)(z)=z,\, \forall z \in E\}\] and
$\liep^{E}=\{v\in \liep:\, [v,E]=0\}$.
   \end{enumerate}
 \end{lemma}
A subalgebra $\lieq \subset \lieg$ is \enf{parabolic} if
$\lieq^\C$ is a parabolic subalgebra of $\lieg^\C$.  One way to
describe the parabolic subalgebras of $\lieg$ is by means of
restricted roots.  If $\lia \subset \liep$ is a maximal subalgebra,
let $\roots(\lieg, \lia)$ be the (restricted) roots of $\lieg$ with
respect to $\lia$, let $\lieg_\la$ denote the root space corresponding
to $\la$ and let $\lieg_0 = \liem \oplus \lia$, where $\liem =
\liez_\liek(\lia)=\liez (\lia)\cap \liek$. We denote by $\mathfrak z (\lia)=\{x\in \lieg:\, [x,\lia]=0\}$.  Let $\simple \subset \roots(\lieg, \lia)$ be a
base and let $\roots_+$ be the set of positive roots. If $I\subset
\simple$, set $\roots_I : = \spam(I) \cap \roots$. Then
\begin{gather}
  \label{para-dec}
  \lieq_I:= \lieg_0 \oplus \bigoplus_{\la \in \roots_I \cup \roots_+}
  \lieg_\la
\end{gather}
is a parabolic subalgebra. Conversely, if $\lieq \subset \lieg$ is a
parabolic subalgebra, then there are a maximal subalgebra $\lia
\subset \liep$ contained in $\lieq$, a base $\simple \subset
\roots(\lieg, \lia)$ and a subset $I\subset \simple $ such that $\lieq
= \lieq_I$.  We can further introduce
\begin{gather}
\label{notaz-I}
\begin{gathered}
  \lia_I : = \bigcap_{\la \in I} \ker \la \qquad \lia^I := \lia_I^\perp \\
  \lien_I = \bigoplus_{\la \in \roots_+ \setminus \roots_I} \lieg_\la
  \qquad \liem_I : = \liem \oplus \lia^I \oplus \bigoplus_{\la \in
    \roots_I}\lieg_\la.
\end{gathered}
\end{gather}
Then $\lieq_I = \liem_I \oplus \lia_I \oplus \lien_I$. Since
$\theta\lieg _ \la = \lieg_{-\la}$, it follows that $ \lieq_I \cap
\theta\lieq_I = \lia_I \oplus \liem_I$.  This latter Lie algebra coincides
with the centralizer of $\lia_I$ in $\lieg$. It is a Levi factor
of $\lieq_I$ and
\begin{gather}
  \label{liaI}
  \lia_I =\liez (\lieq_I \cap \theta\lieq_I) \cap \liep.
\end{gather}
If we denote by $\Delta_{-}$ the set of negative root, then $\mathfrak n_I^{-}= \bigoplus_{\la \in \roots_{-} \setminus \roots_I} \lieg_\la$ is a subalgebra.
It follows from standard commutation relations that $\mathfrak z (\lia_I)$ normalizes $\mathfrak n_I$ and $\mathfrak n_I^{-}$ and the centralizer of $\lia^I$ in either is reduced to zero. Then, keeping in mind $\lieg=\mathfrak n_I^{-}\oplus \lieq_I$, $\lieq_I$ is self-normalizing.
\begin{defin}
A subgroup $Q$ of $G$ is called parabolic if it is the normalizer of a parabolic subalgebra in $\lieg$.
\end{defin}
The normalizer of $\lieq_I$ is the \emph{standard parabolic subalgebra} $Q_I$. Let $R_I$ and let $A_I$ be the unique connected Lie subgroups of $G$ with Lie algebra equals to $\mathfrak n_I$ and $\lia_I$ respectively. $R_I$ is the unipotent radical of  $Q_I$. The group $Q_I$ is the semidirect product of $R_I$ and of $Z(A_I)$, i.e., the centralizer of $A_I=\exp(a_I)$ in $G$. Moreover, $Z(A_I)=A_I \times M_I$, where $M_I$ is a closed Lie group whose Lie algebra is $\mathfrak m_I$. It is not connected in general but it is compatible. Since $M_I$ is stable with respect to the Cartan involution, $K_I=M_I\cap K$ is  maximal compact in $M_I$. It is also maximal compact in $Q_I$ and the quotient
\[
X_I=M_I/K_I=Q_I/ K_I A_I N_I
\]
is a symmetric space of noncompact type for $M_I$. Finally, as a consequence of the Iwasawa decomposition $G=NAK$, where $N=\exp(\mathfrak n)$, $\mathfrak n=\mathfrak n_{\emptyset}$, and $NA\subset Q_I$, and so the following result holds
\begin{prop}\label{parabolic-group}
$G=KQ_I$.
\end{prop}
Another way to describe parabolic subgroups of $G$ is the following.

If $\beta
\in \liep$, the endomorphism $\mathrm{ad}(\beta)$ is diagonalizable over $\R$. Denote by $V_\lambda (\mathrm{ad}(\beta)$
the eigenspace of $\mathrm{ad}(\beta)$ corresponding to the eigenvalue $\lambda$. Set
\begin{gather*}
\lieg^{\beta+}:=\bigoplus_{\lambda \geq 0} V_\lambda (\mathrm{ad}(\beta),\\
\lier^{\beta+}: = \bigoplus_{\la > 0} V_\la (\ad \beta),
\end{gather*}
\begin{gather*}
  G^{\beta+} :=\{g \in G : \lim_{t\to - \infty} \exp({t\beta}) g
  \exp({-t\beta}) \text { exists} \},\\
  R^{\beta+} :=\{g \in G : \lim_{t\to - \infty} \exp({t\beta})
  g \exp({-t\beta}) =e \}. \\
 \end{gather*}
The following result characterizes completely the parabolic subgroups of $G$. The result is classical and a proof is given in \cite{biliotti-ghigi-heinzner-2}.
\begin{lemma}
  $G^{\beta +} $ is a parabolic subgroup of $G$ with Lie algebra $\lieg^{\beta+}$ and it is the semidirect product of $G^{\beta}$ with $R^{\beta+}$.  Moreover, $G^\beta$ is a Levi factor, $R^{\beta+}$ is connected with Lie algebra $\mathfrak r^{\beta+}$ and it is the
  unipotent radical of $G^{\beta+}$. Every parabolic subgroup of $G$ equals
  $G^{\beta+}$ for some $\beta \in \liep$.
\end{lemma}
\subsection{Basic properties of the gradient map}\label{subsection-gradient-moment}
Let $(Z, \omega)$ be a K\"ahler manifold. Assume that $U^\C$ acts
holomorphically on $Z$, that $U$ preserves $\om$ and that there is a
momentum map $\mu: Z \lra \liu$.  If $\xi \in \liu$ we denote by $\xi^\#$
the induced vector field on $Z$, i.e., $\xi^\# (p)=\desudtzero \exp(t\xi) p$, and we let $\mu^\xi \in C^\infty (Z)$ be
the function $\mu^\xi(z) := \langle \mu(z),\xi\rangle$, where $\langle \cdot,\cdot \rangle$ is an $\Ad(U)$-invariant scalar product on $\liu$.  That $\mu$ is the
momentum map means that it is $U$-equivariant and that $d\mu^\xi =
i_{\xi^\#} \omega$.

Let $G \subset U^\C$ be compatible.
If $z \in Z$, let $\mup (z) \in \liep$ denote $-i$ times the component
of $\mu(z)$ in the direction of $i\liep$.  In other words, if we also denote by $\langle \cdot,\cdot \rangle$ the $\mathrm{Ad}(U)$-invariant scalar product on $i\liu$ requiring the multiplication by $i$ is an orthogonal map from $\liu$ onto $i\liu$, then the formula $\langle \mup (z) , \beta \rangle = \langle i\mu(z),\beta \rangle=\langle \mu(z) , -i\beta\rangle$ for any
$\beta \in \liep$ defines the \emph{gradient map}
\begin{gather*}
  \mu_\liep : Z \lra \liep.
\end{gather*}
Let $\mup^\beta \in C^\infty (Z)$ be the function $ \mup^\beta(z) = \langle
\mup(z) , \beta\rangle = \mu^{-i\beta}(z)$.  Let $\metrica$ be the \Keler
metric associated to $\om$, i.e. $(v, w) = \om (v, Jw)$. Then
$\beta^\#$ is the gradient of $\mup^\beta$. If $M \subset Z$ is a
locally closed $G$-invariant submanifold, then $\beta^\#$ is the
gradient of $\mup^\beta \restr{M}$ with respect to the induced
Riemannian structure on $M$. From now on we always assume that $M$ is compact and connected.
\begin{thm}\label{line}[Slice Theorem \protect{\cite[Thm. 3.1]{heinzner-schwarz-stoetzel}}]
  If $x \in M$ and $\mup(x) = 0$, there are a $G_x$-invariant
  decomposition $T_x M = \lieg \cd x \oplus W$, open $G_x$-invariant
  subsets $S \subset W$, $\Omega \subset M$ and a $G$-equivariant
  diffeomorphism $\Psi : G \times^{G_x}S \ra \Omega$, such that $0\in
  S, x\in \Omega$ and $\Psi ([e, 0]) =x$.
\end{thm}
$\ $ \\
Here $G \times^{G_x}S$ denotes the associated bundle with principal
bundle $G \ra G/G_x$.
\begin{cor} \label{slice-cor} If $x \in M$ and $\mup(x) = \beta$,
  there are a $G^\beta$-invariant decomposition $T_x M = \lieg^\beta
  \cd x \, \oplus W$, open $G^\beta$-invariant subsets $S \subset W$,
  $\Omega \subset M$ and a $G^\beta$-equivariant diffeomorphism $\Psi
  : G^\beta \times^{G_x}S \ra \Omega$, such that $0\in S, x\in \Omega$
  and $\Psi ([e, 0]) =x$.
\end{cor}
This follows applying the previous theorem to the action of $G^\beta$
with the gradient map $\widehat{\mu_{\liu^\beta}} := \mu_{\liu^\beta} -
i\beta$, where $\mu_{\liu^\beta}$ denotes the projection of $\mu$ onto
$\liu^\beta$.
See \cite[p.$169$]{heinzner-schwarz-stoetzel} and \cite{sjamaar} for more details.
\begin{cor} \label{slice-cor-2}
  If $\beta \in \liep $ and $x \in M$ is a critical point of $\mupb$,
  then there are open invariant neighborhoods $S \subset T_x M$ and
  $\Omega \subset M$ and an $\R$-equivariant diffeomorphism $\Psi : S
  \ra \Omega$, such that $0\in S, x\in \Omega$, $\Psi ( 0) =x$. Here
  $t\in \R$ acts as $d\phi_t(x)$ on $S$ and as $\phi_t$ on $\Omega$.)
\end{cor}
\begin{proof}
  The subgroup $H:=\exp(\R \beta)$ is compatible.  It is enough to
  apply the previous corollary to the $H$-action at $x$.
\end{proof}
Let $x \in
\Crit(\mu^\beta_\liep)=\{y\in M:\, \beta^\# (y)=0\}$. Let $D^2\mup^\beta(x) $ denote the Hessian,
which is a symmetric operator on $T_x M$. Denote by $V_-$ (respectively $V_+$) the sum of the eigenspaces of the
Hessian of $\mupb$ corresponding to negative (resp. positive)
eigenvalues. Denote by $V_0$ the kernel.  Since the Hessian is
symmetric we get an orthogonal decomposition
\begin{gather}
  \label{Dec-tangente}
  T_x M = V_- \oplus V_0 \oplus V_+.
\end{gather}
Let $\alfa : G \ra M$ be the orbit map: $\alfa(g) :=gx$.  The
differential $d\alfa_e$ is the map $\xi \mapsto \xi^\# (x)$. The following result is well-know. A proof is given in \cite{biliotti-ghigi-heinzner-2}.
\begin{prop}\label{linearization1}
  \label{tangent}
  If $\beta \in \liep$ and $x \in \Crit(\mu^\beta_\liep)$ then
  \begin{gather*}
    D^2\mup^\beta(x) = d \beta^\# (x).
  \end{gather*}
  Moreover $d\alfa_e (\mathfrak r^{\beta\pm} ) \subset V_\pm$ and $d\alfa_e(
  \lieg^\beta) \subset V_0$.  If $M$ is $G$-homogeneous these are
  equalities.
\end{prop}
\begin{cor}\label{MorseBott}
For every $\beta \in \liep$, $\mupb$ is a Morse-Bott function.
\end{cor}
\begin{proof}
Corollary \ref{slice-cor-2} implies that $\Crit(\mu^\beta_\liep)$ is a smooth
    submanifold. Since $T_x \Crit(\mu^\beta_\liep) = V_0$ for $x\in \Crit(\mu^\beta_\liep)$, the
  first statement of Proposition \ref{tangent} shows that the Hessian
  is nondegenerate in the normal directions.
\end{proof}
 Let $g\in G$ and  let $\xi \in \liep$. It is easy check that
\[
(\mathrm{d} g)_p (\xi^\#)=(\mathrm{Ad}(g)(\xi) )^\# (gp ).
\]
Therefore $G^\beta$ preserves $\Crit(\mu^\beta_\liep)$.
 \begin{cor}\label{sugg1}
If $M$ is $G$-homogeneous then $G^\beta$-orbits are open and closed in $ \Crit(\mu^\beta_\liep)$.
\end{cor}
\begin{proof}
Since $T_x \Crit(\mu^\beta_\liep) = V_0=T_x G^\beta \cdot x$ for $x\in \Crit(\mu^\beta_\liep)$, the result follows.
 \end{proof}

Let $c_1 > \cds > c_r$ be the critical
  values of $\mup^\beta$. The corresponding level sets of $\mup^\beta$, $C_i:=(\mup^\beta)\meno ( c_i)$ are submanifolds which are
 union of components of $\Crit(\mup^\beta)$. The function $\mup^\beta$ defines a gradient flow generated by its gradient which is given by $\beta^\#$.
By Theorem \ref{line}, it follows that for any $x\in M$ the limit:
  \begin{gather*}
    \phi_\infty (x) : = \lim_{t\to +\infty }
    \exp(t\beta ) x,
  \end{gather*}
  exists.
Let us denote by $W_i^\beta$ the \emph{unstable manifold} of the critical component $C_i$
  for the gradient flow of $\mup^\beta$:
  \begin{gather}
    \label{def-wu}
    W_i^\beta := \{ x\in M: \phi_\infty (x) \in C_i \}.
  \end{gather}
Applying Theorem \ref{line}, we have the following well-known decomposition of $M$ into unstable manifolds with respect to $\mup^\beta$.
\begin{thm}\label{decomposition}
In the above assumption, we have
\begin{gather}
    \label{scompstabile}
    M = \bigsqcup_{i=1}^r W_i^\beta,
  \end{gather}
and for any $i$ the map:
  \begin{gather*}
    (\phi_\infty)\restr{W_i} : W_i^\beta \ra C_i,
  \end{gather*}
  is a smooth fibration with fibres diffeomorphic to $\R^{l_i}$ where
  ${l_i}$ is the index (of negativity) of the critical submanifold
  $C_i$
\end{thm}
Let $\beta \in \liep$. Proposition \ref{linearization1} implies that $\mathrm{Max}(\beta)=\{x\in M:\, \mup^\beta(x)=\mathrm{max}_{y\in M} \mup^\beta\}$ is a smooth possibly disconnected submanifold of $M$.
\begin{lemma}\label{parabolic-preserve-maximun}
$\mathrm{Max}(\beta)$ is $G^{\beta+}$-invariant.
\end{lemma}
\begin{proof}
$G^\beta$ preserves $\Crit(\mu^\beta_\liep)$. By Proposition
  \ref{linearization1}, $R^{\beta+}$ acts trivially on $\mathrm{Max}(\beta)$. Therefore, it is enough to prove that $G^\beta$ preserves $\mathrm{Max}(\beta)$.

Let $y\in \mathrm{Max}(\beta)$. Since $\mup$ is $K$-equivariant,  $K^\beta$ preserves $\mathrm{Max}(\beta)$. Let $\xi\in \liep^\beta$ and let $\gamma(t)=\exp(t \xi) y$. Since $\beta_\# (\gamma(t))=0$ it follows that $\mup^\beta (\gamma(t) )$ is constant and so $\exp(t\xi)y\in \mathrm{Max}(\beta)$. Hence, keeping in mind that $G^\beta=K^\beta \exp(\liep^\beta)$, the result follows.
\end{proof}
\begin{prop}\label{closed-orbit-parabolic}
$\mathrm{Max}(\beta)$ contains a closed orbit of $G^{\beta+}$ which coincides with a $K^\beta$-orbit.
\end{prop}
\begin{proof}
$(G^{\beta})^o$ preserves any connected component of $\mathrm{Max}(\beta)$. The restriction of $\mup$ on any connected component defines a $(G^{\beta})^o$-gradient map \cite{heinzner-schwarz-stoetzel}. By Corollary 6.11 in \cite{heinzner-schwarz-stoetzel} p. $21$, see also Proposition \ref{heinzner-maximun}, $(G^{\beta})^o$ has a closed orbit which coincides with a $(K^{\beta})^o$-orbit. Since $G^{\beta}$ has a finite number of connected components and any connected component of $G^\beta$ intersects $K^\beta$, it follows that $G^{\beta}$ has a closed orbit which coincides with a $K^{\beta}$-orbit. This orbit is a closed orbit of $G^{\beta+}$ since $R^{\beta+}$ acts trivially on $\mathrm{Max}(\beta)$, concluding the proof.
\end{proof}
Using an $\mathrm{Ad}(K)$-invariant  inner product of  $\liep$, we define $\nu_\liep (z):=\frac{1}{2}\parallel \mup(z)\parallel^2$. The function $\nu_\liep$ is $K$-invariant and it is called \emph{the norm square function}.
The following result is proved in \cite{heinzner-schwarz-stoetzel} (see Corollary 6.11 and Corollary 6.12 p. $21$).
\begin{prop}\label{heinzner-maximun}
Let $x\in M$. Then:
\begin{itemize}
\item if $\nu_\liep$ restricted to $G\cdot x$ has a local maximum at $x$, then $G\cdot x=K\cdot x$
\item if $G\cdot x$ is compact, then $G\cdot x =K\cdot x$
\end{itemize}
\end{prop}
A strategy to analyzing the $G$-action on $M$ is to view $\nu_\liep$ as generalized Morse function. In \cite{heinzner-schwarz-stoetzel} the authors proved the existence of a smooth $G$-invariant stratification of $M$ and they studied its properties.
\subsection{Satake compactifications}\label{Satake-section}
In this section we always refer to \cite{borel-ji-libro,gjt,Satake}.

Let $G$ be a real noncompact connected semisimple Lie group $G$ with finite center and let $\lieg=\liek \oplus \liep$ be a Cartan decomposition of its Lie algebra. Let $K$ be a maximal compact subgroup of $G$ with Lie algebra $\liek$ and let $\tau$ be an irreducible representation of $G$ with finite kernel on a complex vector space $V$. We also assume there exists a $K$-invariant Hermitian product $h$ on $V$ such that $\tau(G)\subset \mathrm{U}(V,h)^\C$ is compatible.
With these data Satake \cite{Satake} constructed a compactification $\XS$
by the symmetric space $X=G/K$. A good references of symmetric spaces is \cite{helgason}.    We wish
to recall the construction of the Satake compactifications and some of
their relevant properties.  Proofs
can be found in the book \cite[\S I.1]{borel-ji-libro}, see also \cite{gjt}, which we follow for most of the notation.

Put
\begin{gather*}
  \Herm (V) = \{ A\in \End(V) : A=A^*\},
\end{gather*}
where $A^*$ denotes the adjoint of $A$ with respect to $h$.
Denote by $\pi : \Herm(V)\setminus \{0\} \ra \PP(\Herm(V))$ the
canonical projection and set
\begin{gather*}
  \Pos(V) = \pi(\{A\in \Herm (V) : A >0\}) \subset \PP(\Herm(V)).
\end{gather*}
$\Pos(V)$ consists of points $[A]$ such that $A$ is invertible and all
its eigenvalues have the same sign. The following result is easy to check.
\begin{lemma}
  \label{prop-PV}
  (a) $ \overline { \Pos(V) } =\pi ( \{A\in \Herm (V) : A\neq 0 ,
  A\geq 0\})$.  (b) The restriction of $\pi$ to $ \{A\in \Herm (V) : A
  >0, \det A=1\}$ is a homeomorphism onto $\Pos(V)$.  (c) The
  restriction of $\pi$ to $ \{A\in \Herm (V) : A \geq 0, \tr A=1\}$ is
  a homeomorphism onto $\overline{\Pos(V)}$.
\end{lemma}
\begin{defin}
  For $G,K, \tau, h $ as before, set
  \begin{gather}
    \label{eq:imbeddoinpos}
    i_\tau : X: =G/K \ra \Pos(V) \qquad i_\tau(gK) =[ \tau(g)\tau(g)^*].
  \end{gather}
  The \enf{Satake compactification} of $X$ associated to $\tau$ and
  $\scalo$ is the space $ \XS:=\overline {i_\tau (X) } $.  The closure
  is taken in $\PP(\Herm(V))$.
\end{defin}
Since $\Sl(V)$ and hence $G$ acts on $\PP(\Herm(V))$ by
\[
g\cdot [A]:=[gAg^*],
\]
$\XS$ is a $G$-compact\-if\-ication.  We stress that $\XS$ depends only on
$G,K, \tau$ and $h$. Satake gave a thorough description of the boundary $\partial\XS :=\XS
- i_\tau (X)$ in terms of root data.

Let $\roots(\lieg, \lia)$ be the (restricted) roots of $\lieg$ with
respect to $\lia$. Let $\Delta(\lieg,\lia)$ be the set of simple roots in $\roots(\lieg,\lia)$ determined by the positive chamber $\lia^{+}$. Let $\mu_\tau$ be the \emph{highest weight} of $\tau$ with respect to the partial ordering determined by $\Delta$. If $\lambda$ is another weight of $\tau$, then it has the form
\[
\lambda=\mu_\tau-\sum_{\alpha \in \Delta(\lieg,\lia)} c_{\alpha,\lambda} \alpha,
\]
where $c_{\alpha,\lambda}$ are non negative integers. The support of $\lambda$ is the
set $ \supp (\lambda) : = \{\alfa \in \simple: c_{\alpha,\lambda} >0\}$.
\begin{defin}
  A subset $I\subset \simple$ is $\mu_\tau$-\enf{connected} if
  $I\cup\{\mu_\tau \}$ is connected, i.e., it is not the union of subsets orthogonal with respect to the Killing form.
\end{defin}
Connected components of $I$ are defined as usual.  The $I$ is $\mu_\tau$-connected subset if and only if any connected component of $I$ contains at least one element $\alpha$  which is not orthogonal to $\mu_\tau$.
\begin{lemma}
  [\protect{\cite[Lemma 5 p. 87]{Satake}}] $I\subset
  \simple$ is $\mu_\tau$-connected if and only if $I=\supp (\lambda)$ for
  some weight $\lambda$ of $\tau$.
\end{lemma}
For example $\emptyset = \supp(\mu_\tau)$ and $\simple$ is
$\mu_\tau$-connected since $\tau$ is nontrivial on any simple factor of
$G$.  Given $\lambda$ a weight of $\tau$, we denote by $V_\lambda$ the corresponding eigenspace.

Let $I$ be a $\mu_\tau$-connected subset of $\Delta(\lieg,\lia)$ and let
\begin{gather*}
  V_I=\bigoplus_{\supp(\lambda) \subset I} V_\lambda.
\end{gather*}\label{def-VI-SI}
Let $Q_I$ be the normalizer of $\mathfrak q_I$, see Section \ref{compatible-parabolic}. $Q_I$ is a parabolic subgroup of $G$ given by
\[
Q_I=N_I A_I M_I.
\]
$M_I$ is a Levi factor of $Q_I$ that it is non connected in general.
\begin{lemma}
  [{\cite[Lemma 8 p. 89]{Satake}}]
  \label{lemmetto-di-Satake-taui}
 The subspace  $V_I$ is invariant under $\tau(g)$, $g\in Q_I$ and the induced representation of $M_I$ on $V_I$, denoted $\tau_I : M_I \ra \Gl(V_I)$, is a multiple of an irreducible faithful one.
\end{lemma}
\begin{defin}
    If $I\subset \simple$ is $\mu_\tau$-connected, denote by $I'$ the
    collection of all simple roots orthogonal to $\{\mu_\tau \}\cup I$.
    The set $J:=I\cup I'$ is called the $\mu_\tau$-\enf{saturation} of
    $I$.
\end{defin}
$I$ is the largest $\mu_\tau$-connected subset contained in $J$. Observe that union of $\tau$-connected subspaces is a $\tau$-connected subspace. Then $I$ is the union of all $\tau$-connected subspaces contained in $J$ and so the largest $\mu_\tau$-connected of $J$ is unique.
\begin{lemma}
    [{\cite[Prop. I.4.29 p. 70]{borel-ji-libro}}]
    \label{stab-VI}
    If $I$ is $\mu_\tau$-connected, then $ Q_J=\{g\in G: \tau(g)V_I=V_I\}$.
  \end{lemma}
  Fix a $\mu_\tau$-connected subset $I$. If $A\in \End (V_I)$, let
  $A\oplus 0$ denote the extension of $A$ that is trivial on
  $V_I^\perp$.  If $\pi_I: V\ra V_I$ denotes orthogonal projection and
  $j_I: V_I \hookrightarrow V$ denotes the inclusion, then $A\oplus 0
  =j_I\circ A\circ \pi_I$ and the map
  \begin{gather}
    \label{eq:def-psi}
    \psi_{I} : \PP(\Herm (V_I)) \ra \PP(\Herm(V)) \qquad \psi_I ([A]) =
    ([A\oplus 0])
  \end{gather}
  embeds $\Pos(V_I) $ in $\overline{\Pos(V)}$. Note that $K_I=M_I \cap
  K$ is a maximal compact subgroup of $M_I$  and $X_I=M_I/K_I$ is again a symmetric
  space of noncompact type.  Hence we have a map $i_{\tau_I} : X_I \ra
  \Pos(V_I)$ defined as in \eqref{eq:imbeddoinpos}.  Finally define
  \begin{gather*}
    i_I=\psi_I \circ i_{\tau_I} : X_I \ra \overline{\Pos(V)}.
  \end{gather*}
  \begin{thm}
    [{\cite[Cor. I.4.32]{borel-ji-libro}}]\label{Satakone}
    \begin{gather*}
      \XS = \bigsqcup_{\text{$\mu_\tau$-connected $I$}} G .\, i_I(X_I).
    \end{gather*}
  \end{thm}
  If $I=\simple$ then $i_I(X_I) = X$. The sets $g.i_I(X_I)$ with $g\in
  G$ and $I\subsetneq \simple$ are called \emph{boundary components}.
\begin{lemma}
  [{\cite[Prop. I.4.29]{borel-ji-libro}}] \label{Satakotto} The
  boundary components are disjoint: if $I$ is $\mu_\tau$-connected and
  $g\in G$ then $g.\, i_I(X_I) \cap i_I(X_I) \neq \emptyset$ if and
  only if $g.\, i_I(X_I) = i_I(X_I) $ if and only if $g\in Q_J$.
\end{lemma}
\section{Finite dimensional representations and gradient map}\label{representation-gradient map}
Let $G$ be real noncompact Lie group with finitely many connected components. $G$ admits maximal compact subgroups which are conjugate under the identity component $G^o$ \cite{helgason,knapp-beyond}. In this section we always assume that $G$ is reductive, connected, linear and the center of $G$ is compact. If we denote by $K$ a maximal compact subgroup of $G$, then  $G$ admits a Cartan involution $\theta$ with the fixed points set containing $K$. Since the center of $G$ is compact then it is contained in $K$ \cite{knapp-beyond}. We have the classical Cartan decomposition
\[
\lieg =\liek \oplus \liep,
\]
with $[\liek,\liek]\subset \liek$, $[\liek,\liep]\subset \liep$ and $[\liep,\liep]\subset \liek$. The restriction of the Killing form of $\lieg$ on $\liep$ defines a $G$-invariant metric on $G/K$ and with this metric $G/K$ is a simply connected complete symmetric space of noncompact type, i.e., of nonpositive curvature \cite{helgason}.

Let $V$ be a finite dimensional vector space over $\C$ and let $\tau: G \lra \mathrm{GL}(V)$ be a representation. We say that $\tau$ is irreducible if it has no non trivial invariant subspaces. The representation of $\lieg$ associated by differentiation with a representation $\tau$ of $G$ is called the tangent representation associated with $\tau$ and it is also denoted by $\tau$. By \cite[4.32 Proposition]{gjt}, there exists a Hermitian scalar product $h$ on $V$ such that $\tau(G)\subset \mathrm{U}(V,h)^\C$ is compatible, where $\mathrm{U}(V,h)$ is the unitary group of $V$ with respect to the Hermitian scalar product $h$. This means that $\tau(K)\subset \mathrm{U}(V,h)$, $\tau(G)=\tau(K)\exp(\tau(\liep))$ with $\tau(\liep)\subset i \mathrm{Lie}(\mathrm{U}(V,h)$.
From now on we identify $G$ with $\tau(G)$. The Zariski closure of $G$ in $\mathrm{U}(V,h)^\C$ is given by $U^\C$, where $U$ is a compact connected subgroup of $\mathrm{U}(V,h)$  and $G$ is compatible with respect to the Cartan decomposition of $U^\C$ \cite[Lemma 1 $p.3$]{heinz-stoezel}. Moreover, keeping in mind that $G$ has a compact center, by  Propositions $1$ and $2$ in \cite[$p.4$]{heinz-stoezel}, there exist compact connected Lie subgroups $U_0$ and $U_1$ of $U$ which centralize each other, such that
\begin{enumerate}
\item $U^\C=U_0^\C \cdot U_1^\C$ and the intersection $U_0 \cap U_1$ is a finite subgroup of the center of $U$;
\item $G=G_0\cdot U_1^\C$, where $G_0$ is a real form of $U_0^\C$ which is compatible with the Cartan decomposition of $U_0^\C$. Moreover, $\mathfrak g_0=\mathfrak k_0 \oplus \mathfrak p_0$, $\mathfrak u_0=\mathfrak k_0 \oplus i\mathfrak  p_0$ and $G_0 =K_0 \exp (\liep_0)$.
\end{enumerate}
Let $\mu:\mathbb P(V) \lra \liu$ be the momentum map of $U$ and let $\mup:\mathbb P(V) \lra \liep$ be the $G$-gradient map associated to $\mu$. Let $\beta \in \liep$ and let
\[
\mathrm{Max}(\beta)=\{x\in \mathbb P(V):\, \mup^\beta (x)=\mathrm{max}_{z\in \mathbb P(V)}\, \mup^\beta  \}.
\]
Let $\lambda_1>\dots> \lambda_k$ be the eigenvalues of $\beta$.  We denote by $V_1,\dots, V_k$ the corresponding eigenspaces. In view of the orthogonal decompositions $V=V_1\oplus\dots\oplus V_{k}$, $\mup^\beta$ is given by
\[
\mup^\beta ([x_1+\cdots+x_k])=\frac{\lambda_1 \parallel x_1 \parallel^2 + \cdots +\lambda_k \parallel x_k \parallel^2}{\parallel x_1 \parallel^2 + \cdots +\parallel x_k \parallel^2}.
\]
Therefore $\mathrm{Max}(\beta)=\mathbb P(V_1)$. Since the gradient flow of $\mup^\beta$ is given by
\[
\R \times \mathbb P(V)\lra \mathbb P(V), \qquad (t,[x_1+\cdots + x_k])\mapsto [e^{t \lambda_1}x_1+\cdots + e^{t \lambda_k} x_k],
\]
the critical points of $\mup^\beta$ are union of proper projective subspaces, $\mathbb {P}(V_1)\cup\dots\cup \mathbb {P}(V_{k})$, of $\mathbb P(V)$.

Set $\Gamma=\exp(\R\beta)$. We recall that a submanifold $M\subset \mathbb P(V)$ is called \emph{full} if it is not contained in any proper projective linear subspace of $\mathbb P(V)$.
\begin{lemma}\label{unstable-maxima-restriction}
Let $M\subset \mathbb P(V)$ be a full $\Gamma$-invariant closed subset. Then
$
\mathrm{Max} (\beta) \cap M \neq \emptyset.
$
\end{lemma}
\begin{proof}
It is easy to check that the unstable manifolds of $\mup^\beta$ are given by:
$$
W_1^\beta=\mathbb {P} (V)\setminus \mathbb {P}(V_2\oplus\dots\oplus V_k),
$$
$$
W_2^\beta=\mathbb {P}(V_2\oplus\dots\oplus V_k)\setminus \mathbb {P}(V_3\oplus\dots\oplus V_k),
$$
$$
\vdots
$$
$$
W_{k-1}^\beta=\mathbb {P}(V_{k-1}\oplus V_k)\setminus \mathbb {P}(V_k),
$$
$$
W_k^\beta=\mathbb {P}(V_k).
$$
The unstable manifold of the global maximum is the complement of a proper projective subspace of $\mathbb P(V)$ and so it is open and dense. In particular $\mup^\beta$ has a unique local maximum which is the global maximum.
Since $M$ is full, it follows $M\cap W_1^\beta \neq \emptyset$. Let $y\in M\cap W_1^\beta$. Then, keeping in mind that $M$ is closed and $\Gamma$-invariant, we have
\[
\lim_{t\mapsto +\infty}  \exp(t\beta ) y \in \mathrm{Max}(\beta) \cap M
\]
and so the result follows.
\end{proof}
Let $\lia\subset \liep$ be an Abelian subalgebra and let $A=\exp(\lia)$. The $A$-gradient map is given by  $\mua=\pi_\lia \circ \mup$, where $\pi_\lia$ is the orthogonal projection of $\liep$ onto $\lia$. It is well-known that $\mua(\mathbb P(V))$ is a polytope. Moreover $\mua(\mathbb P(V)^A)$ is finite and $\mathrm{conv}(\mua(\mathbb P(V)^A))=\mua(\mathbb P(V))$ \cite{atiyah-commuting,guillemin-sternberg-convexity-1}, where $\mathbb P(V)^A=\{x\in \mathbb P(V):\, A\cdot x =x\}$ is the set of fixed point of $A$.
\begin{prop}\label{image-gradient}
Let $M\subset \mathbb P(V)$ be a full $A$-invariant closed subset such that $\mua(M)$ is convex. Then $\mua(\mathbb P(V))=\mua(M)$.
\end{prop}
\begin{proof}
Denote by $C_1=\mua(M)$ and by $C_2=\mua(\mathbb P(V))$. By the previous Lemma, for any $\beta \in \lia$, we have
\[
\mathrm{max}_{y\in C_2} \langle y , \beta \rangle=\mathrm{max}_{z\in  \mathbb P(V)} \mua^\beta=\mathrm{max}_{z\in M} \mua^\beta =\mathrm{max}_{y\in C_1} \langle y,\beta \rangle.
\]
Since $C_1\subseteq C_2$, applying Proposition \ref{convex-criterium} we get $C_1=C_2$.
\end{proof}
\begin{lemma}\label{closed-G-orbit}
Let $v\in \mathbb P(V)$ be such that $U^\C \cdot v$ is closed. Then there exists a unique closed orbit of $G$ contained in $U^\C \cdot v$.
\end{lemma}
\begin{proof}
Assume that $U^\C \cdot v$ is closed. It is well-known that $U^\C \cdot v =U\cdot v$ is a flag manifold \cite{guillemin-sternberg-convexity-1}. By \cite[$p.197$]{goodman-wallach-Springer} there are irreducibly representations $\psi_i : U_0^\C \lra \mathrm{GL}(W_i)$ and
$\phi_i : U_1^\C \lra \mathrm{GL}(L_i)$, for $i=1,\ldots,k$ such that
$\mathbb P(V)=\bigoplus_{i=1}^k \mathbb P(W_i \otimes L_i)$. Moreover $v=[v_i\otimes w_i] \in \mathbb P (V_i\otimes W_i)$ for some $i\in\{1,\ldots,k\}$, where $v_i \in W_i$, respectively $w_i \in L_i$, is a highest weight vector of $\psi_i$, respectively $\phi_i$, and $U^\C \cdot [v_i\otimes w_i]$ is the unique closed orbit of $U^\C$ on $\mathbb P (W_i \otimes L_i)$ \cite{huckleberry-introduction-DMV}.
Then the compact orbit $U^\C \cdot [v_i\otimes w_i]$ fibers $U^\C$-equivariantly over $U_0^\C \cdot [v_i]$ with fiber $U_1^\C \cdot [w_i]$. The result now follows directly by the Wolf's result \cite{wolf}.
\end{proof}
From now on, we always assume $G$ is semisimple and $\tau:G \lra \mathrm{PGL}(V)$ is irreducible. $\tau$ induces a projective representation of $U^\C$ on $\mathbb P(V)$ which is also irreducible. By Borel-Weil Theorem, the $U^\C$-action on $\mathbb P(V)$ has a unique closed orbit that we denote by $\mathcal O'$ \cite{huckleberry-introduction-DMV}. By the above Lemma there exists a unique closed $G$-orbit contained in $\mathcal O'$ that we denote by $\mathcal O$. By Proposition \ref{heinzner-maximun}, $\mathcal O$ is a $K$-orbit. We claim there exists a link between the projective representation $\tau:G \lra \mathrm{PSL}(V)$, the  $G$-gradient map  and $\mathcal O$. The next two results are useful.
\begin{prop}\label{restriction-closed-orbit}
Let $\lia$ be an Abelian subalgebra contained in $\liep$ and let $A=\exp(\lia)$. Let $\beta \in \lia$ and let $\mua:\mathbb P(V) \lra \lia$ be the $A$-gradient map. The following items hold true:
\begin{enumerate}
\item[(a)] $\mathrm{Max} (\beta)=\mathbb P(W)$, where $W$ is the eigenspace associated to the maximum of $\beta$;
\item[(b)] $\mathrm{Max}(\beta) \cap \mathcal O$ is not empty, and so
\[
\mathrm{Max}(\beta) \cap \mathcal O=\mathrm{Max}_{\mathcal O}(\beta)=\{z\in \mathcal O: \mup^\beta(z)=\mathrm{max}_{y\in \mathcal O}\, \mup^\beta \};
\]
\item[(c)] the unstable manifold relative to the maximum of $\mup^\beta :\mathcal O \lra \R$, is given by $\mathcal O \setminus  \mathbb P (W^\perp )$;
\item[(d)]
$
\mua(\mathcal O)=\mua(\mathbb P(V));
$
\item[(e)] if $\lia$ is a maximal Abelian subalgebra contained in $\liep$, then
$\mup(\mathbb P(V))$ is the convex hull of a Weyl group orbit.
\end{enumerate}
\end{prop}
\begin{proof}
Denote by $\pi: V \setminus\{0\} \lra \mathbb P(V)$ the natural projection. Let $w\in V\setminus\{0\}$ such that $G\cdot \pi(w)=\mathcal O$. Since the $G$-action on $V$ is irreducible, it follows that $G\cdot w$ is not contained in any subspace of $V$. This means that $\mathcal O$ is full in $\mathbb P(V)$. Applying Lemma \ref{unstable-maxima-restriction} and Proposition \ref{image-gradient}, items $(a),(b),(c), (d)$ hold. Assume that $A=\exp(\lia)$, where  $\lia$ is a maximal Abelian subalgebra contained in $\liep$. Since $\mathcal O$ is a $K$-orbit and  $\mup$ is $K$-invariant, keeping in mind that $\mua=\pi_\lia \circ \mup$, applying a Theorem of Kostant \cite{kostant-convexity},  we get $\mua(\mathcal O)$ is the convex hull of a Weyl group orbit $\mathcal W$, concluding the proof.
\end{proof}
\begin{cor}\label{orbitope}
$
\mathrm{conv}(\mup(\mathbb P(V)))=\mathrm{conv}(\mup(\mathcal O)),
$
and so it is a polar orbitope. In particular any face of $\mathrm{conv}(\mup(\mathbb P(V)))$ is exposed.
\end{cor}
\begin{proof}
Let $\lia \subset \liep$ be a maximal Abelian subalgebra contained in $\liep$.
The convex hull of $\mup(\mathbb P(V))$ is a $K$-invariant compact convex subset of $\liep$ satisfying
\[
\pi_\lia (\mathrm{conv}(\mup(\mathbb P(V)))=\mua(\mathbb P(V)).
\]
By Proposition \ref{restriction-closed-orbit}, we have $\mua(\mathbb P(V))=\mua(\mathcal O)$.
Applying \cite[Theorem 0.1 p. $424$]{bgh-israel-p}, we have
\[
\mathrm{conv}(\mup(\mathcal O))=K\mua(\mathcal O)=\mathrm{conv}(\mup(\mathbb P(V))).
\]
This implies that $\mathrm{conv}(\mup(\mathbb P(V)))$ is a polar orbitope. By \cite[Theorem 3.2 p. $597$]{biliotti-ghigi-heinzner-2}, every face of $\mathrm{conv}(\mup(\mathbb P(V)))$  is exposed, concluding the proof.
\end{proof}
In the sequel we denote by $\mathcal E=\mathrm{conv}(\mup(\mathbb P(V)))$. If $\lia \subset \liep$ is a maximal Abelian subalgebra, then we denote by $P=\mua(\mathbb P(V))$.

We shall investigate the action of a parabolic subgroup of $G$ on $\mathbb P(V)$. We first briefly discuss the complex case.

Let $Q\subset U^\C$ be a parabolic subgroup. $Q$ is connected. \cite{akhiezer}.
\begin{lemma}\label{lemma-parabolico-complesso}
If $Q$ is a parabolic subgroup of $U^\C$, then $Q$ has a unique closed orbit which is contained in $\mathcal O'$. This orbit is a complex $U\cap Q$-orbit, so a flag manifold, and it coincides with the maximum of a contraction of the momentum map.
\end{lemma}
\begin{proof}
Let $v\in \mathbb P(V)$ be such that $Q\cdot v$ is closed. By Proposition \ref{parabolic-group}, $Q\cdot v$ is contained in $\mathcal O'$, which is a flag manifold. Let $\xi \in \liu$ be such that $Q=(U^\C)^{i\xi+}$.  In
\cite[Proposition 3.9]{biliotti-ghigi-heinzner-1} the authors proved that
\[
\mathrm{Max}_{\mathcal O'} (\xi):=\{x\in \mathcal O':\, \mu^\xi (x)=\mathrm{max}_{z\in \mathcal O'}\mu^{\xi}  \}
\]
is the unique compact orbit of $Q$. It is a complex $U\cap Q$-orbit, and so it is connected, and a flag manifold.
\end{proof}
Now, we consider the real case. We start with the following Lemma.
\begin{lemma}\label{local-global}
Let $\beta \in \liep$. Then any local maximum of $\mup^\beta$ restricted to $\mathcal O$ is a global maximum of $\mup^\beta:\mathbb P(V) \lra \R$.
\end{lemma}
\begin{proof}
Let $x\in \mathcal O$ be a local maximum of $\mup^\beta : \mathcal O \lra \R$. The $G$-gradient map restricted to $\mathcal O$
\[
\mup: K\cdot x \lra K \cdot \mup (x),
\]
is a smooth fibration. Hence $\mup(x)$ is a local maximum of the height function
\[
K\cdot \mup(x) \lra \R, \qquad z \mapsto \langle z,\beta\rangle.
\]
By \cite[proof of Proposition 3.1, p. $583$]{biliotti-ghigi-heinzner-2}, it follows that $\mup(x)$ is a global maximum. Since
\[
\mathrm{Max}_{p\in \mathcal O} \mup^\beta=\mathrm{Max}_{z\in K\cdot \mup(x)} \langle \cdot,\beta\rangle,
\]
it follows that $x$ is a global maximum of $\mup^\beta:\mathcal O \lra \R$. By Lemma \ref{unstable-maxima-restriction}, $x$ is a global maximum of $\mup^\beta:\mathbb P(V) \lra \R$ and the result is proved.
\end{proof}
Firstly, we assume that $G$ has a unique closed orbit $\mathcal O$ in $\mathbb P(V)$. By Lemma \ref{closed-G-orbit}, $\mathcal O\subset \mathcal O'$.
\begin{thm}\label{unique-parabolic}
Let $Q$ be a parabolic subgroup of $G$. Then $Q$ has a unique closed orbit in $\mathbb P(V)$. This orbit is connected and it is contained in $\mathcal O$.
\end{thm}
\begin{proof}
By Proposition \ref{closed-orbit-parabolic}, $Q$ has a closed orbit.  By Proposition \ref{parabolic-group}, it is contained in   $\mathcal O$. We prove this orbit is unique.

Let $\beta \in \liep$ such that $Q=G^{\beta+}$. The proof is given in three steps.\\
$\ $ \\
\noindent $1.$ Any closed $G^{\beta+}$-orbit is contained in $\mathrm{Max}_{\mathcal O} (\beta)$. \\
$\ $

Let $\mathcal Q$ be a closed orbit of $G^{\beta+}$.
We may assume that $\mathcal Q =G^{\beta+}\cdot v$ and $v$ is a local maximum $\mup^\beta$ on $\mathcal Q$.  By Theorem \ref{line}, $R^{\beta+}$ acts trivially on $v$. Hence, by Proposition \ref{linearization1}, keeping in mind that $\mathcal O$ is $G$ homogeneous, $v$ is a local maximum of $\mup^\beta$ on $\OO$. By Lemma \ref{local-global}, $v$ is a global maximum on $\OO$. By Lemma \ref{parabolic-preserve-maximun}, we have $\mathcal Q \subset \mathrm{Max}_{\mathcal O} (\beta)$. \\
$\ $ \\
\noindent $2.$ If $x\in \mathrm{Max}_{\mathcal O}(\beta)$, then   $Q\cdot x$ is closed.

Since $R^{\beta+}$ acts trivially on $\mathrm{Max}_{\mathcal O}(\beta)$, it follows $Q\cdot x=G^{\beta}\cdot x \subset \mathrm{Crit}(\mup^\beta)$. By Corollary \ref{sugg1},
$Q\cdot x$ is closed. We thank the referee for pointing out this short argument.\\
$\ $ \\
\noindent $3.$ $Q$ has a unique closed orbit and it is connected. \\
$\ $ \\
\noindent By definition of the gradient map, for any $\xi \in\liep$, we have
\[
\langle \mu (z), -i\xi \rangle=\langle \mup(z), \xi \rangle.
\]
By Lemma \ref{unstable-maxima-restriction}, we have
\[
\mathrm{Max}_{\mathcal O'} (-i\beta) \cap \mathcal O = \mathrm{Max}_{\mathcal O} (\beta).
\]
By Lemma \ref{lemma-parabolico-complesso}, $\mathrm{Max}_{\mathcal O'} (-i\beta)$ is connected and the unique closed orbit of $(U^\C)^{\beta+}$ on $\mathbb P(V)$. By Lemma \ref{closed-G-orbit}, $(G^{\beta})^o$ has a unique closed orbit on $\mathrm{Max}_{\mathcal O'} (-i\beta)$. By step $2$, it follows that $G^{\beta+}$ has a unique closed orbit. This orbit is connected and it coincides with $\mathrm{Max}_{\mathcal O} (\beta)$.
\end{proof}
\begin{cor}\label{unique-parabolici-explicitely}
Let $\beta \in \liep$. Then $\mathrm{Max}_{\mathcal O} (\beta)$  is the unique closed orbit of $G^{\beta+}$ and it is a $(K^\beta)^o$-orbit.
\end{cor}
\begin{proof}
 In the above Theorem, we prove that $\mathrm{Max}_\mathcal O (\beta)$ is the unique compact orbit of $G^{\beta+}$. It is connected and it is a $(G^{\beta})^o$-orbit. By Proposition \ref{heinzner-maximun}, it is a $(K^{\beta})^o$-orbit.
\end{proof}
The unique closed orbit of a parabolic subgroup of $G$ is well-adapted to $\tau$.
\begin{thm}\label{parabloic-representation1}
Let  $\beta\in \liep$ and let $\lambda_1>\dots> \lambda_k$ be its eigenvalues.  We denote by $V_1,\dots, V_k$ the corresponding eigenspaces. Then
\begin{itemize}
\item $G^{\beta+}$ preserves $V_1$. Moreover, $V_1$ is the unique subspace of $V$ on which $G^{\beta+}$ acts irreducibly;
\item $\mathrm{Max}_\mathcal O (\beta) =\mathbb P(V_1) \cap \OO$ is the unique closed orbit of $G^{\beta+}$;
\end{itemize}
\end{thm}
\begin{proof}
$\mathbb P(V_1)=\mathrm{Max}_{x \in \mathbb P(V)} \mup^\beta$. By Proposition \ref{parabolic-preserve-maximun}, we get $G^{\beta+}$ preserves $\mathbb P(V_1)$ and so $V_1$. We claim that $G^{\beta+}$ acts irreducibly on $V_1$. Since $R^{\beta+}$ acts trivially on $\mathbb P (V_1)$, and $G^\beta$ is compatible,   the representation of $G^{\beta+}$ on  $V_1$ splits. Assume that
$
V_1=L\oplus Z.
$
By Proposition \ref{heinzner-maximun}, $G^{\beta}$ has a closed orbit in both $\mathbb P(L)$ and $\mathbb P(Z)$ respectively. Therefore $G^{\beta+}$ admits two closed orbits which is a contradiction. Now, we prove the uniqueness.

Assume that  $G^{\beta+}$ acts irreducibly on $W$. Since $[\liep^{\beta+},\mathfrak r^{\beta+}]\subset \mathfrak r^{\beta+}$, by Engel Theorem \cite{knapp-beyond}, it follows   $R^{\beta+}$ acts irreducibly on $W$ and so it acts trivially on $\mathbb P(W)$. By Proposition \ref{heinzner-maximun}, $\mathbb P(W)$ contains the unique closed orbit of $G^{\beta+}$ which is full in $\mathbb P(V_1)$. This implies $\mathbb  P (V_1 )\subset \mathbb P (W)$ and so $V_1=W$. Applying Corollary  \ref{unique-parabolici-explicitely}, we get  that $\mathrm{Max}_\mathcal O (\beta) =\mathbb P(V_1) \cap \OO$ is the unique closed orbit of $G^{\beta+}$
\end{proof}
Now, we prove the general case.

Let $Q=G^{\beta+}$ be a parabolic subgroup. By Lemma \ref{closed-G-orbit}, there exists a unique closed $G$-orbit on $\mathcal O'$, that we denote by $\OO$. The proof of Theorem \ref{unique-parabolic} shows that $Q$ has a unique closed orbit contained in $\mathcal O$. This orbit is connected and is given by
\[
\mathbb P (V_1 ) \cap \mathcal O=\{z\in \mathcal O:\, \mup^\beta (z)=\mathrm{max}_{y\in \mathcal O} \mup^\beta\},
\]
where $V_1$ is the eigenspace associated to the maximum of $\beta$. Moreover,
\[
\mathbb P(V_1)=\{p\in \mathbb P(V):\, \mup^\beta (p)=\mathrm{max}_{y\in \mathbb P(V)}\mup^\beta\} =\{p\in \mathbb P(V):\, \mu^{-i\beta}(p)=\mathrm{max}_{z\in \mathbb P(V)} \mu^{-i\beta}\}.
\]
By Theorem \ref{parabloic-representation1}, $V_1$ is the unique subspace of $V$ such that $(U^{\C})^{\beta+}$ acts irreducibly on it. By the linearization Theorem, $R^{\beta+}$ acts trivially on $\mathbb P(V_1)$. The Lie algebra of $G^\beta$ is a real form of the Lie algebra of $ (U^{\C})^{\beta}$. Since $(U^{\C})^{\beta+}$ is connected, keeping in mind that $V_1$ is a complex subspace, $G^{\beta}$, and so $G^{\beta+}$, acts irreducibly on $V_1$.

Let $W\subset V$ be a subspace such that $G^{\beta+}$ acts irreducibly on it. If $\beta=\beta_0+i\beta_1 \in \liep_0\oplus i \liu_1$, then $G^{\beta+}=G_0^{\beta_0+} \cdot (U_1^\C)^{i\beta_1}$.

Let  $Q$ be the parabolic subgroup of $U_0^\C$ such that its Lie algebra is the complexification of the Lie algebra of $G_0^{\beta_0+}$. Then $Q\cdot (U_1^\C)^{i\beta_1+}$ preserves $W$ and acts irreducibly on it. This implies that the unipotent radical of $Q\cdot (U_1^\C)^{i\beta_1+}$ acts trivially on $\mathbb P(W)$. By Proposition \ref{heinzner-maximun}, $Q$ has a closed orbit on $\mathbb P(W)$. By Lemma \ref{lemma-parabolico-complesso} this orbit is unique and it is contained in $\mathcal O'$. Now, the unipotent radical of $G^{\beta+}$ acts trivially on the unique closed orbit of
$Q \cdot (U_1^\C)^{i\beta_1+}$ on $\OO'$. By Lemma \ref{closed-G-orbit}, $(G^\beta)^o$ has a unique closed orbit contained in the unique closed orbit of $Q \cdot (U_1^\C)^{ i\beta_1+}$ on $\OO'$. Hence $G^{\beta+}$ has a unique closed orbit contained in $\OO'$, and so in $\OO$, and this orbit is connected. By Theorem \ref{parabloic-representation1}, the unique compact orbit of the $G^{\beta+}$-action on  $\OO$ is given by $\mathbb P(V_1) \cap \OO$. Therefore  $V_1\subset W$ and so $V_1=W$.
Summing up, we have proved the following result
\begin{thm}\label{parabloic-representation}
Let  $\beta\in \liep$ and let $\lambda_1>\dots> \lambda_k$ be its eigenvalues.  We denote by $V_1,\dots, V_k$ the corresponding eigenspaces. Then
\begin{itemize}
\item $G^{\beta+}$ preserves $V_1$. Moreover, $V_1$ is the unique subspace of $V$ on which $G^{\beta+}$ acts irreducibly;
\item $\mathbb P (V_1 ) \cap \mathcal O=\{z\in \mathcal O:\, \mup^\beta (z)=\mathrm{max}_{y\in \mathcal O} \mup^\beta\}$ is the unique closed orbit of $G^{\beta+}$ contained in $\mathcal O$. Moreover, it is connected and a $(K^{\beta})^o$ orbit;
\end{itemize}
\end{thm}
\begin{remark}\label{g-invariance}
Let $Q$ be a parabolic subgroup of $G$ and let $\mathcal O=K\cdot x$. Although $\mup$ is not $G$-equivariant,
if $Q\cdot x$ is closed, then $Q\cdot \mup(x) \subset K\cdot \mup(x)$ is closed and
\[
\mup(Q\cdot x)=Q\cdot \mup(x).
\]
Indeed, pick $\beta \in \liep$ such that $Q=G^{\beta+}$. By Corollary \ref{unique-parabolici-explicitely}, we have
\[
\mup(G^{\beta+}\cdot x)=
\{z\in \mathcal \mup(\mathcal O):\, \langle z, \beta \rangle=\mathrm{max}_{r\in \mup(\mathcal O)} \langle r, \beta\rangle   \}.
\]
By \cite[Corollary 3.1 p. $593$ and Proposition 3.9 p. $599$]{biliotti-ghigi-heinzner-2}, we have
\[
\{z\in \mathcal \mu(\mathcal O):\, \langle z, \beta \rangle=\mathrm{max}_{r\in \mup(\mathcal O)} \langle r, \beta\rangle   \}=K^{\beta}\cdot \mup(x)=G^{\beta+}\cdot \mup(x).
\]
\end{remark}
We summarize the link between the parabolic subgroups of $G$ and the $G$-gradient map.

Let $\beta \in \liep$ and let $Q=G^{\beta+}$. Let $\mathcal O (Q)$ and $W(Q)$ denote the unique closed orbit of $Q$ contained in $\mathcal O$ and the unique irreducible submodule of $Q$, respectively. $W(Q)$ is the eigenspace associated  to the maximum of $\beta \in \liep$ and
\[
\mathbb P(W(Q))\cap \mathcal O=\mathrm{Max}_{\mathcal O} (\beta)=\mathcal O(Q).
\]
By Remark \ref{g-invariance} and Theorem 1.1 \cite[p. $582$]{biliotti-ghigi-heinzner-2}, we have
$
\mup (\mathcal O (Q))=\mathrm{ext}\, F_\beta (\mathcal E).
$
Since $\mup^{-1}(F_\beta (\mathcal E))=\mathbb P (W(Q))$, it follows
\[
\mathcal O(Q)=\mup^{-1}(F_\beta (\mathcal E)) \cap \mathcal O.
\]

Let $W\subset V$ be a complex subspace. We denote by
\[
\pi_W : V \lra W,
\]
the orthogonal projection and by
\[
\hat\pi_W : \mathbb P(V) \setminus{\mathbb P(W^\perp)}  \lra \mathbb P(V), \qquad \hat\pi_W ([v])=[\pi_W (v)],
\]
its projectivization. $\hat\pi_W$ is a meromorphic map. Since $\mathbb P(W(Q))=\mathrm{Max}(\beta)$, by Theorem \ref{scompstabile} and Proposition \ref{parabloic-representation}, the domain of the map $\hat\pi_{W(Q)}$ is the unstable manifold of the maximum of $\mup^{\beta}$. Moreover,  $\hat\pi_W$ coincides with $\phi_\infty$. Indeed, let $\lambda_1>\dots> \lambda_k$ be the eigenvalues of $\beta$. Let $V_2,\ldots,V_k$ be the eigenspaces associated to $\lambda_2,\ldots,\lambda_k$. Then
\[
\begin{split}
\lim_{t\mapsto +\infty} \exp(t\beta) [x_1+x_2+\cdots+x_k]&=\lim_{t\mapsto +\infty} [e^{t\lambda_1}x_1+e^{t\lambda_2}x_2+\cdots +e^{t\lambda_k}x_k]
\\ &=\lim_{t\mapsto +\infty} [x_1+e^{t(\lambda_2-\lambda_1)}x_2+\cdots +e^{t(\lambda_k-\lambda_1)}x_k]\\ &=[x_1] \\ &=\hat\pi_W (x).
\end{split}
\]
By Theorem \ref{decomposition}, an unstable manifold flows onto the corresponding critical set. Hence, keeping in mind that
$
\mathcal O (Q)=\mathrm{Max}_{\mathcal O} (\beta),
$
we get $\hat\pi_{W(Q)} (\mathcal O )=\mathcal O(Q)$. Summing up, we have proved the following result.
\begin{thm}\label{flow-parabolic}
Let $Q=G^{\beta+}$ be a parabolic subgroup of $G$. Then
\begin{enumerate}
\item $\mup(\mathcal O (Q))=\mathrm{ext}\, F_\beta (\mathcal E)$;
\item $\mup^{-1}( F_\beta (\mathcal E))=\mathbb P (W(Q))$;
\item $\mathbb P(W(Q)) \cap \mathcal O=\mathcal O (Q)$;
\item $\hat\pi_{W(Q)} (\mathcal O )=\mathcal O(Q )$.
\end{enumerate}
\end{thm}
\section{$\tau$-connected subspaces, parabolic subgroups of $G$ and gradient map}\label{satake-new-description}
In this section we explicitly determine, up to $K$-equivalence, the irreducible representations of parabolic subgroups of $G$ induced by $\tau$.

Given a $W\subseteq V$, $W\neq \{0\}$, set
$
Q(W):=\{g\in G:\, g(W)=W\}.
$
\begin{defin}
$W$ is a $\tau$-connected subspace if $Q(W)$ is parabolic and acts irreducibly on $W$.
\end{defin}
Let $\mathcal O(Q(W))$ be the unique closed orbit of $Q(W)$ contained in $\mathcal O$. By Theorem \ref{flow-parabolic}, we have
\[
\mup(\mathcal O(Q(W)))=\mathrm{ext}\, F_W,
\]
where $F_W \in \mathscr F (\mathcal E)$.
Now, $\mathcal E=\mathrm{conv} ( \mup(\mathcal O))$ and $\mup(\mathcal O)$ is a $K$-orbit. The $K$-action extends to a $G$ action on $\mup(\mathcal O)$ \cite{heinzner-stoetzel-global}. The set of the extreme points of $F_W$ is contained in $\OO$. We define
\[
Q_{F_W}=\{h\in G:\, h\, \mathrm{ext}\, F_W=\mathrm{ext}\, F_W \},\qquad H_F= K\cap Q_{F_W}.
\]
\begin{prop}\label{stabilizzatore-faccie}
$Q_{F_W}=Q(W)$. Moreover, if $\beta \in \mathrm{C}_{F_W}^{H_F}$, then $G^{\beta+}=Q(W)$. Hence $Q(W)$ only depends on the face $F_W$.
\end{prop}
\begin{proof}
Let $\beta\in \mathrm{C}_{F_W}^{H_F}$. By Theorem \ref{flow-parabolic},  $\mup^{-1}(F_\beta (\mathcal E))=\mathbb P(W)$ and $W(G^{\beta+})=W$. This implies  $G^{\beta+}\subseteq Q(W)$. Since $G^{\beta+}=Q_{F_W}$ \cite[Proposition 3.8, $p.598$]{biliotti-ghigi-heinzner-2}, it follows that $Q_{F_W}\subseteq Q(W)$.

Let $g\in Q(W)$. By Proposition \ref{parabolic-group}, $g=kp$ for some $k\in K$ and some $p\in Q_{F_W}$. Then $gW=W$ implies $kW=W$ and so,  keeping in mind that $\mathcal O(Q(W))=\mathbb P (W)\cap \mathcal O$, $k$ preserves $\mathcal O(Q(W))$.
By the $K$-equivariance of $\mup$, we get $k\, \mathrm{ext}\, F_W=\mathrm{ext}\, F_W$ and so $k\in H_F\subset Q_{F_W}$. This proves $Q(W)\subseteq Q_{F_W}$, concluding the proof.
\end{proof}
A $\tau$-connected subspace $W$ is completely determined by the closed orbit $\mathcal O (Q(W))$.
\begin{prop}\label{orbita-sottospazio}
Let $W_1,W_2 \subseteq V$ be two $\tau$-connected subspaces. Then $W_1 \subseteq W_2$ if and only if $\mathcal O (Q(W_1))\subseteq \mathcal O(Q(W_2))$.
\end{prop}
\begin{proof}
The linear span of $\mathcal O(Q(W_1))$ is $\mathbb P(W_1)$. Hence, if $\mathcal O (Q(W_1))\subseteq \mathcal O(Q(W_2))$, then $\mathbb P(W_1) \subseteq \mathbb P(W_2)$.

If $W_1 \subseteq W_2$, then $\mathbb P (W_1 ) \cap \mathcal O =\mathcal O (Q(W_1))\subseteq \mathbb P(W_2) \cap \mathcal O=\mathcal O (Q(W_2))$ and the result follows.
\end{proof}
\begin{cor}\label{tau-faccie}
Let $W_1,W_2 \subseteq V$ be $\tau$-connected subspaces. Then $W_1\subseteq W_2$ if and only if $F_{W_1}\subseteq F_{W_2}$. Moreover $W_1=W_2$ if and only if $\mathrm{relint}\, F_{W_1} = \mathrm{relint}\, F_{W_1}$.
\end{cor}
\begin{proof}
If $W_1\subseteq W_2$, then $\mathcal O(Q(W_1))=\mathbb P(W_1) \cap \mathcal O\subseteq \mathbb P (W_2) \cap \mathcal O=\mathcal O(Q(W_2))$, and so
\[
\mathrm{ext}\, F_{W_1}=\mup(\mathcal O(Q(W_1)))\subseteq \mup(\mathcal O (Q(W_2)))=\mathrm{ext}\, F_{W_2}.
\]
Therefore $F_{W_1}=\mathrm{conv}\, (\mathrm{ext}\, F_{W_1}) \subseteq \mathrm{conv}\, ( \mathrm{ext}\, F_{W_2})=F_{W_2}$.
Vice-versa, if $F_{W_1}\subseteq F_{W_2}$, then
\[
\mathcal O(Q(W_1))=\mup^{-1}(F_{W_1} ) \cap \mathcal O \subseteq \mup^{-1} (F_{W_2 })\cap \mathcal O=\mathcal O(Q(W_2 )).
\]
By Proposition \ref{orbita-sottospazio}, we get $W_1\subseteq W_2$. The last item follows from Theorem \ref{schneider-facce}. Indeed, $F_{W_1}=F_{W_2}$ if and only if $\mathrm{relint}\, F_{W_1}=\mathrm{relint} F_{W_2}$.
\end{proof}
\begin{remark}
Let $F\subseteq \mathcal E$ be a face. By Lemma \ref{face-chain}, there exists a maximal  chain
\[
F=F_0 \subsetneq F_1\subsetneq \cdots \subsetneq F_k=\mathcal E
\]
of faces. By Theorem \ref{flow-parabolic},  $\mup^{-1}(F_i)=\mathbb P(W_i)$ and $W_i$ is a $\tau$-connected subspaces of $V$. By Corollary \ref{tau-faccie}, we get a chain
\[
\mathcal O(Q(W_0))\subsetneq \mathcal O( Q(W_1)) \subsetneq \cdots \subsetneq \mathcal O (Q(W_{k-1})) \subsetneq \mathcal O
\]
of homogeneous submanifolds of $\mathbb P(V)$.
\end{remark}
Set $\mathcal H (\tau)=\left\{(W,Q(W)): W\, \mathrm{is\ } \tau-\mathrm{connected\ subspace\ of\ }V \right\}$. The following Lemma is easy to check.
\begin{lemma}\label{action-tauconnected}
Let $W$ be a $\tau$-connected subspace of $V$ and let $g\in G$. Then
\begin{enumerate}
\item $gW$ is a $\tau$-connected subspace;
\item there exists $k\in K$ such that $gW=kW$;
\item $Q(gW)=gQ(W)g^{-1}=kQ(W)k^{-1}$
\end{enumerate}
\end{lemma}
$G$ acts on $\mathcal H (\tau)$ as follows:
\[
g(W,Q(W)):=(gW,gQ(W)g^{-1}).
\]
\begin{prop}\label{tau-connected-face}
The map
\[
\mathscr{Z} : \mathcal H (\tau) \lra \mathscr{F} ( \mathcal E), \qquad (Q,Q(W)) \mapsto F_W.
\]
is $K$-equivariant and bijective.
\end{prop}
\begin{proof}
By Lemma \ref{action-tauconnected}, the map $\mathscr{Z}$ is $K$-equivariant. By Theorem \ref{flow-parabolic} and Corollary \ref{tau-faccie}, the map $\mathscr Z$ is bijective.
\end{proof}
Let $\lia\subset \liep$ be a maximal Abelian subalgebra and let $P=\mua(\OO)=\mathcal E \cap \lia$.  Fix a system of root of simple
roots $\Pi \subset \Delta = \Delta (\lieg, \lia)$.  We denote by $\lia_+$ the positive Weyl chamber associated to $\Pi$. Now,  $\mup(\OO) \cap \lia_+=\{x\}$ and by Kostant convexity Theorem \cite{kostant-convexity}, $P=\mathrm{conv}(\mathcal W \cdot x)$.

A subset $B\subset
\lia^*$ is \emph{connected} if there is no pair of disjoint subsets
$D,C\subset B$ such that $D\sqcup C =B$, and $D$ and $C$ are two subsets orthogonal with respect to the Killing Cartan form. Connected components are defined as
usual.  If $x $ is a nonzero vector of $\lia$, a subset $I
\subset\Pi$ is called $x$-\enf{connected} if any connected component of $I$ contains at least one root
$\alfa$ such that $\alfa (x) \neq 0$.  If $I\subset \Pi$ is
$x$-connected, denote by $I'$ the collection of all simple roots
orthogonal to $\{ x\}\cup I$, i.e., if $\alpha \in I'$, then $\alpha$ is orthogonal to $I$ and satisfies $\alpha(x)=0$.  The set $J:=I\cup I'$ is called the
$x$-\enf{saturation} of $I$.  The largest $x$-connected subset
contained in $J$ is $I$. So $J$ is determined by $I$ and $I$ is
determined by $J$.  Given a subset $I\subset \simple$ we denote
by $Q_I$ the parabolic subgroup with Lie algebra $\lieq_I$ as defined
in \eqref{para-dec}.
In \cite{biliotti-ghigi-heinzner-2}, see also \cite{biliotti-ghigi-heinzner-1,kobert-Scheiderer}, the following theorem is proved.
\begin{thm}\label{parabolici-faccie}
Let $x\in \lia_+ $ be a point such that $P=\mathrm{conv}(\mathcal W \cdot x)$. Then
\begin{enumerate}
\item Let $I \subset \Pi$ be a $x$-connected subset and let $J$ be its $x$-saturation. Then $Q_I\cdot x=Q_J \cdot x$ and $F:=\mathrm{conv}(Q_I\cdot x)$ is a face of $\mathcal E$. Moreover,
$F=F_\beta (\mathcal E)=\mathrm{conv}(K^{\beta} \cdot x)$, where $\beta \in \lia$ satisfies $Q_I=G^{\beta+}$;
\item  let $\beta'\in \lia$ be such that $Q_J=G^{\beta'+}$. Then $F=F_{\beta'} (\mathcal E)=\mathrm{conv}(K^{\beta'}\cdot x)$. Moreover $\{h\in G:\, h\,\ext{F}=\ext{F}\}=Q_J=G^{\beta'+}$;
\item Any face of  $\mathrm{Conv}(\mathcal E)$ is conjugate to one of the faces constructed
    in (a).
\end{enumerate}
\end{thm}
In the sequel, we denote by $F_I$ the face of $\mathcal E$ such that $\mathrm{ext}\, F_I=Q_I \cdot x$. Let  $W_I=\mup^{-1}(F_I)$. By Proposition \ref{stabilizzatore-faccie}, $Q(W_I)=Q_J$.
Applying Proposition \ref{tau-connected-face} and Theorem \ref{parabolici-faccie}, we have the following result.
\begin{cor}\label{tau-connected-mutau-connected}
Let $W$ be a $\tau$-connected. Then there exists $k\in K$ and a $\{x\}$-connected subset $I$ of $\Pi$ such that $W=kW_I$. Moreover  $Q(W)=kQ_Jk^{-1}$, where $J$ is the saturation of $I$
\end{cor}
Now, we prove the following result.
\begin{thm}\label{rappresentazioni-parabolici}
 In the above setting, the map
 \[
\mathscr X : I \mapsto (Q_I,W_I),
 \]
 induces a bijection  between $\{x\}$-connected subsets and the irreducible representations of parabolic subgroups of $G$ induced by $\tau$, up to the $K$-action.
\end{thm}
\begin{proof}
Let $I_1,I_2 \subset \Pi$ be $x$-connected subsets. Let  $J_1$ and $J_2$ be their $x$-saturations. Assume that there exits $k\in K$ such that $Q_{I_1}\cdot x=k Q_{I_2}\cdot x$. Then $Q_{J_1}=kQ_{J_2}k^{-1}$. By Lemma I.2.1.11 in \cite{warner}, see also \cite[Proposition 2.18, pag.$20$]{gjt}, $J=J_1=J_2$ and $k\in Q_{J}$. Since $I_1$, respectively, $I_2$ is the maximal $\tau$-connected subset of $J$, it follows that $I_1=I_2$. This proves $\mathscr X$ is injective.

Let $Q$ be a parabolic subgroup of $G$ and let $W$ be the unique complex subspace of $V$ be such that $Q$ acts irreducibly on $W$. By Corollary \ref{tau-connected-mutau-connected}, there exists $k_1\in K$ such that $W=k_1W_I$ and $Q  {\color{red} \subseteq}  k_1 Q_J k_1^{-1}$. By Theorem \ref{parabloic-representation},  $\mathcal O(Q)=\mathcal O (k_1 Q_j k^{-1})$. Hence, keeping in mind that $\mathcal O(k_1 Q_J k_1^{-1})=k_1 \mathcal O (Q_J)$, it follows that
\[
\mathrm{ext}\, F_I=k_1^{-1} \mathrm{ext}\, F_W
\]
By \cite[Proposition 3.5]{biliotti-ghigi-heinzner-2}, there exists $k_2 \in K$ such that  $\lia \subset k_2 \mathfrak q k_2^{-1}$ and  $(k_2 \mathrm{ext}\, F_W) \cap \lia$ is a face of $P$. By the main result  proved in \cite{biliotti-ghigi-heinzner-2} $\mathscr F (P)/\mathcal W \cong \mathscr F (\mathcal E) /K$. Hence, there exists $\theta \in N_K (\lia)$ such that
\[
\mathrm{ext}\, F_I \cap \lia=(\theta k_2 \mathrm{ext} F_W)\cap \lia.
\]
By \cite[Theorem 1.1]{biliotti-ghigi-heinzner-2}, we have $\mathrm{ext}\, F_I=\theta k_2 \mathrm{ext} F_W$.

Set $k=\theta k_2$. Since $\lia \subset k\lieq k^{-1}$, there exists a $\tilde J$ subset of $\Pi$ such that $Q_{\tilde J}=kQk^{-1}$. By Remark \ref{g-invariance}, we have
\[
Q_{\tilde J} \cdot x=Q_I \cdot x=\ext F_I.
\]
Let $\tilde I$ denote the maximal $\{x\}$-connected subset contained in $\tilde J$. By Theorem \ref{parabolici-faccie}, $Q_{\tilde I} \cdot x$ is a face and $Q_{E}\cdot x=Q_{\tilde I}\cdot x$, where $E$ is the saturation of $\tilde I$. Since $\tilde I \subseteq \tilde J \subseteq E$, it follows
\[
Q_{\tilde I}\cdot x =Q_{\tilde J} \cdot x=Q_{E} \cdot x.
\]
On the other hand $Q_I\cdot x= Q_{\tilde J} \cdot x$ and so, by Theorem \ref{parabolici-faccie}, $Q_E =Q_{J}$.
This implies $E=J$ and  $\tilde I=I$  due to the fact that $I$, respectively $\tilde I$, is the maximal $\{x\}$-connected subspace of $J$. Moreover,
\[
Q_I\subset Q_{\tilde J}\subset Q_J,
\]
where $I\subseteq \tilde J \subseteq J$.

Let $I'=\tilde J \setminus\{I\}$. The set $I'$ is perpendicular to $I$ and so the Langlands decomposition of $Q_{\tilde J}$ can be written as
\[
Q_{\tilde J}=N_{I} A_I M_I M_{I'},
\]
see \cite{borel-ji-libro}. By \cite[Lemma I.4.25, p. $69$]{borel-ji-libro}, $M_{I'}$ acts trivially on $W_I$. Hence, the $Q_{\tilde J}$-action on $W_I$ is completely determined by the $Q_I$-action on $W_I$.
\end{proof}
\begin{cor}\label{piop}
Let $Q$ be a parabolic subgroup and let $W$ be the unique complex subspace of $V$ such that $Q$ acts on irreducibly on $W$. Then there exists $k\in K$ and a $\{x\}$-connected subset $I$ of $\Pi$ such that $W=kW_I$ and
$
Q_I \subset kQk^{-1} \subset Q_J.
$
\end{cor}
\begin{ese}\label{calcolo}
Let $\mathrm{SL}(3,\R) \circlearrowleft \C^3$ as usual. We choice $\langle \xi , \nu \rangle=\mathrm{Tr}(\xi \nu)$ as $\mathrm{Ad (SO)}(3)$-invariant scalar product on $\mathfrak{sym}_0 (3)$. Let
\[
\lia=\left\{ \left[\begin{array}{ccc} t & & \\ & s &  \\ & & -t-s \end{array}\right], t,s \in \R   \right\}
\]
be a maximal Abelian subalgebra of $\mathrm{sym}_0 (3)$. Then
\[
A=\exp(\lia)=\left\{ \left[\begin{array}{ccc} e^t & & \\ & e^s &  \\ & & e^{-t-s} \end{array}\right], t,s \in \R \right\}
\]
and the $A$-gradient map on $\mathbb P (\C^3)$ is given by
\[
\begin{small}
\mu_\lia ([x,y,z])=
\left[\begin{array}{ccc}
\frac{\parallel x \parallel^2}{ \parallel x \parallel^2 +\parallel y \parallel^2 + \parallel z\parallel^2} -\frac{1}{3}& & \\
& \frac{\parallel y \parallel^2}{ \parallel x \parallel^2 +\parallel y \parallel^2 + \parallel z\parallel^2} -\frac{1}{3} & \\
& & \frac{\parallel z \parallel^2}{ \parallel x \parallel^2 +\parallel y \parallel^2 + \parallel z\parallel^2} -\frac{1}{3}
\end{array}\right]
\end{small}
\]
By the Abelian convexity Theorem \cite{atiyah-commuting, guillemin-sternberg-convexity-1},
$\mu_\lia (\mathbb P (\C^3))$  is the convex hull of the image of the fixed point set of $A$. It is easy to check $\mathbb P (\C^3)^{A}=\left\{ [e_1],[e_2],[e_3]\right\}$, where $\{e_1,e_2,e_3\}$ is the canonical basis of $\C^3$. Therefore $\mu(\mathbb P (\C^3))$ is the convex hull of
\[
\begin{small}
x_1=\left[\begin{array}{ccc} 2/3 & & \\ & -1/3 & \\ & & -1/3 \end{array}\right], x_2=\left[\begin{array}{ccc} -1/3 & & \\ & 2/3 & \\ & & -1/3 \end{array}\right], x_3=\left[\begin{array}{ccc} -1/3 & & \\ & -1/3 & \\ & & 2/3 \end{array}\right]
\end{small}
\]
and so the image is an equilateral triangle.
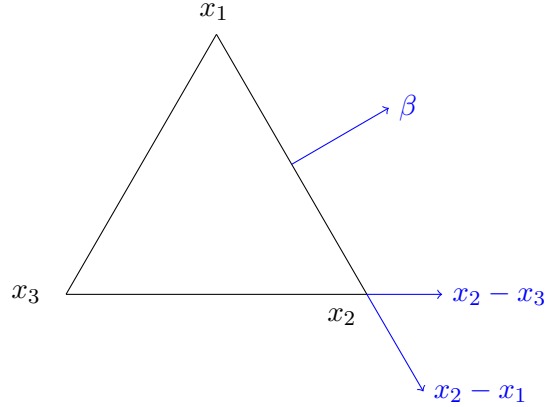
\begin{center}
\begin{tikzpicture}\label{figura}
\draw[-] (-2,0) -- (2,0);
\node[left,black] at (2,-0.3){$x_2$};
\node[left,black] at (-2.2,0) {$x_3$};
\node[left,black] at (0.3,3.76) {$x_1$};
\draw[-] (0,3.46) -- (-2,0);
\draw[-] (0,3.46) -- (2,0);
\draw[->][blue] (2,0)--(3,0)node[right]{$x_2-x_3$};
\draw[->][blue] (2,0)--(2.75,-1.29)node[right]{$x_2-x_1$};
\draw[->][blue] (1,1.73)--(2.29,2.48)node[right]{$\beta$};
\end{tikzpicture}
\end{center}
The proper faces containing $x_2$ which are not Weyl equivalent are: $\{x_2\}$ and the side $x_2-x_1=\left[\begin{array}{ccc} -1 & & \\ & 1 & \\ & & 0 \end{array}\right]$. The segment $x_2-x_1$ defines $\{x_2\}$ as exposed face and
\[
\mu_\lia^{x_2-x_1}([x,y,z))=\mathrm{Tr}(\mu([x,y,z]) (x_2-x_1))=\frac{-\parallel x \parallel^2 + \parallel y \parallel^2 }{\parallel x \parallel^2 + \parallel y \parallel^2 + \parallel z \parallel^2}.
\]
Therefore $\mathrm{Max}(x_2-x_1)=\mathbb P (V_2)$, where $V_2=\mathrm{Span}(e_2)$. Note that the vector $x_2-x_3=\left[\begin{array}{ccc} 0 & & \\ & 1 & \\ & & -1 \end{array}\right]$  defines as well $\{x_2\}$ as an exposed face,
\[
\mu_\lia^{x_2-x_3}([x,y,z])=\frac{\parallel y \parallel^2 - \parallel z \parallel^2 }{\parallel x \parallel^2 + \parallel y \parallel^2 + \parallel z \parallel^2},
\]
and indeed $\mathrm{Max}(x_2-x_3)=\mathbb P (V_2)$. Finally, the face corresponding of the segment $x_2-x_1$ is defined by
\[
\beta=\left[\begin{array}{ccc} 1 & & \\ & 1 & \\ & & -2 \end{array}\right],
\]
and
\[
\mu_\lia^{\beta}([x,y,z))=\frac{\parallel x \parallel^2 + \parallel y \parallel^2 -2\parallel z \parallel^2 }{\parallel x \parallel^2 + \parallel y \parallel^2 + \parallel z \parallel^2}.
\]
Therefore $\mathrm{Max}(\beta)=\mathbb P (V_1)$, where $V_1=\mathrm{Span}(e_1,e_2)$.
\end{ese}
\section{Boundary components of Satake compactifications and gradient map}\label{Satake-sempre-Satake}
In this section we reformulate the Satake's analysis in more geometrical terms. The elements of the Satake compactification associated to $\tau$ are interpreted as rational maps of $\mathcal O$. From now on we always assume that $\tau:G \lra \mathrm{SL}(V)$ is irreducible and the kernel is finite. We also always refer to \cite{borel-ji-libro,gjt,Satake} and we follow the notation introduced in Section \ref{compatible-parabolic}.
\subsection{$\mu_\tau$-connected subspaces, projections and rational maps}
Let $\lia \subset \liep$ be a maximal Abelian subalgebra. Then
\[
\mathfrak z (\lia)=\mathfrak m \oplus \lia,
\]
where $\mathfrak m = \mathfrak z (\lia)\cap \mathfrak k$. If $\lia'\subset \mathfrak m$ is a maximal Abelian subalgebra of $\mathfrak m$, then $\lia'+ i \lia \subset \mathfrak u=\liek + i \liep$ is a maximal Abelian subalgebra of $\liu$ and so $(\lia'+ i \lia)^\C \subset \lieg^{\C}=\liu^{\C}$ is a Cartan subalgebra. Given $\lia, \lia'$,  and $\Pi\subset \Delta(\lieg,\lia)$ be a basis one can choose
a basis of $(\lia + i\lia')^*$ adapted to $\Pi$  and $(i\lia')^*$. Indeed, it is possible to define a set of simple roots
of $\hat{\Delta}$ such that the projection of a subset of the simple roots of $\hat{\Delta}$ onto $\mathfrak a^*$ is equal to $\Pi$  (see \cite[$p. 51-52$]{gjt} and \cite[$p. 272-273$]{helgason}). In particular the Borel subalgebra of $\lieg^{\C}$ is contained in
$q_{\emptyset}^\C=(\mathfrak m + \lia + \mathfrak n)^{\C}$.

Let $\tilde \mu_\tau$ the highest weight of $\lieg^{\C}$ with respect to the partial ordering determined $\hat{\Delta}$. Let $x_o=[v_\tau]$, where $v_\tau$ is any highest weight vector. It is well-known that
\[
\mu:\mathbb P(V) \lra \lia'\oplus i \lia, \qquad \langle \mu(x_o) , \xi \rangle )=\tilde{\mu_\tau}(\xi),
\]
see \cite{biliotti-ghigi-American} that has opposite sign convection for $\mu$, and \cite{baston-eastwood}. By Proposition \ref{restriction-closed-orbit}, $P=\mua(\mathbb P(V))=\mua(\OO)$.
Now, $G\cdot x_o$ is closed \cite[4.29 Theorem $p.59$]{gjt}, hence $G\cdot x_o =\mathcal O$, and
\[
\langle \mua(x_o),\xi \rangle= i \tilde \mu_\tau (\xi).
\]
$(i\tilde{\mu_\tau})_{\vert{\lia}}$ is the highest weight of $\lieg$ with respect the induced order on $\lia^*$.
From now on, we denote by $ \mu_\tau=(i\tilde{\mu_\tau})_{\vert{\lia}}$.
\begin{prop}\label{weight}
The weights of $\tau$ are contained in the convex hull of the Weyl group orbit $\mathcal W \cdot \mu_\tau$.
\end{prop}
\begin{proof}
Denote the weights of $\tau$ by $\mu_1,\ldots,\mu_p \in \lia^*$. It is easy to check that
\[
\mua(\mathbb P(V))=\mathrm{conv}(v_1,\ldots,v_p),
\]
where $v_1,\ldots,v_p \in \lia$ satisfy
\[
\langle v_i,H\rangle = \mu_i (H),
\]
for any $H\in \lia$ an for $i=1,\ldots,n$. Since $\mu_\tau$ lies in the positive Weyl chamber, applying Proposition \ref{restriction-closed-orbit}, we have
\[
\mua(\mathbb P(V))=\mua(\mathcal O)=\mathrm{conv}(\mathcal W \cdot z_\tau),
\]
where $z_\tau \in  \lia$ satisfies
\[
\langle z_\tau, H \rangle =\mu_\tau (H),
\]
concluding the proof.
\end{proof}
The boundary components of $\XS$ has been described in terms of root data in the pioneering work of Satake \cite{Satake}. More precisely, given a $\mu_\tau$-connected subspace $I\subset \Pi$, Satake defined
\[
V_I=\bigoplus_{\supp(\lambda) \subset I} V_\lambda.
\]
By Lemmata \ref{lemmetto-di-Satake-taui} and \ref{stab-VI}, we have $W_I=V_I$, $Q(W_I)=Q_J$ and $M_I$ acts irreducibly on $W_I$. By Theorem \ref{Satakone} the image of the map
\[
i_{\tau_I} : M_I /K\cap M_I \lra \overline{\mathcal P (V)}, \qquad g(K\cap M_I) \mapsto [\tau_I (g(K\cap M_I)\oplus 0],
\]
lies in $\XS$. By Theorem \ref{Satakone} and Theorem \ref{rappresentazioni-parabolici},   boundary components of $\XS$ arise from the faces of $P$ containing $z_\tau$ and which are not equivalent with respect to the Weyl group. Summing up, we have proved the following result.
\begin{thm}\label{satake-geometric}
Let $\XS$ be the Satake compactification associated to $\tau$. The $G$-action on $\XS$ has a finite number of orbits. This number coincides with the number of the faces of  $P=\mua(\mathbb P(V))$ containing $z_\tau$ which are not equivalent with respect to the Weyl group.
\end{thm}
\begin{cor}
 The $G$-action on $\partial \XS$ is transitive if and only if $G/K$ has rank one.
\end{cor}
We describe the boundary components of $\XS$ using the $G$-gradient map restricted on $\OO$.

Let $(\lia,\Pi)$ be a root datum and let $\mu_\tau$ the highest weight of $\tau$. If $I\subset \Pi$ is a $\mu_\tau$-connected subset, then $W_I$ is a $\tau$-admissible subspace and $Q(W_I)=Q_J$ where $J$ is the saturation of $I$.

Let $W$ be  $\tau$-admissible. By Corollary \ref{tau-connected-mutau-connected}, there exists $k\in K$ and a $\mu_\tau$-connected subset $I$ such that $W=kW_I$ and $Q(W)=kQ_Jk^{-1}$, where $J$ is a saturation of $I$.
Set $M(W)=k M_J k^{-1}$ and $K(W)=M(W)\cap K=k K_J k^{-1}$. We claim that $M(W)$ does not depend on $k\in K$

If $kW_I=\tilde kW_I$, then $\tilde k^{-1}k W_I=W_I$ and so $\tilde k^{-1}k\in K\cap Q_J$. Therefore
\[
k M_J k^{-1}=\tilde kM_J \tilde k^{-1}.
\]
Since $M_J$ is compatible, it follows that $M(W)$ is compatible. By Theorem \ref{parabloic-representation}, the representation $\tau_W:=(\tau_{\vert{M(W)}})_{\vert{W}}$  is irreducible. Set $X_W=M(W)/K(W)$.
$X(W)$ is a symmetric space of noncompact type \cite{borel-ji-libro}. We split $V=W\oplus W^\perp$ and we define
\[
i_{\tau_W}:X_W \lra \mathbb  P(\mathcal H (W)), \qquad i_{\tau_W}(gK(W)) \mapsto [\tau_W (g) \tau_W (g)^*],
\]
which induces the following map
\[
\psi_{W}:  X_W \lra  \overline{\mathcal P(V)}, \qquad gK(W) \mapsto [\tau_W (g) \tau_W (g)^* \oplus 0].
\]
Assume that $W=W_I$. Since $J=I\cup I'$ and $I'$ is orthogonal to $I$, it follows that
\[
M_J=M_I M_{I'},
\]
and they commute. Since $M_{I'}$ acts trivially on $W_I$, it follows that $i_{W_I}$ depends only of $M_I /M_I \cap K$ and it is a boundary component of $\XS$  \cite{borel-ji-libro}.  The  same holds for any $\tau$-admissible subspace.  Applying
Theorem \ref{Satakone} and  Lemma  \ref{Satakotto}, we get the following result.
\begin{thm}\label{Satake-mutauconnessi}
The boundary of $\XS$ are exactly the subset $\XS$ of the form $i_W(X_W)$ for some $\tau$-connected subspace $W\subseteq V$, while $X=i_V(V)$. Hence
\[
\XS=\bigsqcup_{W\ \tau-connected} i_W(X_W),
\]
\end{thm}
We shall think  the elements  of $\XS$ as rational maps of $\mathcal O$.

Since $G\subset \mathrm{SU}(V,h)^\C$ is compatible,
\[
\mathcal M=\{g\in G:\, g=g^{-1}\},
\]
is a submanifold and $\exp:\liep \lra \mathcal M$ is a diffeomorphism with inverse $\log:\mathcal M \lra \liep$. Then the map
\[
g \mapsto \sqrt{g}=\exp\left(\frac{\log(g)}{2}\right),
\]
is a diffeomorphism. Note that $\sqrt{q}$ is the Hermitian positive isomorphism of the polar decomposition of $g$ has isomorphism of $V$.
We recall the following elementary fact. A proof is given \cite[Lemma $4.4$, $p.251$]{biliotti-ghigi-American}
\begin{lemma}
\label{sqrt-cont}
Let $V$ be a Hermitian vector space and set $ \mathcal S =\{A\in \mathcal H(V):
A\geq 0\} $. If $A\in \mathcal S$ there is a unique
$B\in \mathcal S $ such that $B^2=A$. Set $\sqrt{A}:=B$. Then
$\sqrt{\cdot}$ is a homeomorphisms of $\mathcal S$ onto itself.
\end{lemma}
Set
\begin{equation}\label{root}
\rho:G \lra \mathcal M, \qquad g \mapsto \sqrt{gg^*}.
\end{equation}
Then $a=g\rho(g)^{-1} \in K$ and $g=a\rho(g)$ is the polar decomposition of $g$.

Let $W\subset V$ be a $\tau$-connected subspace. Let $g\in M(W)$. We define
\[
R_{gW} :\mathcal O \dashrightarrow \mathcal O, \qquad  x \mapsto [\sqrt{\tau(g)\tau(g)^*}\pi_W (x)]
\]
\begin{lemma}\label{immagine-mappa-razionale}
In the above notation, the following hold:
\begin{enumerate}
\item[a] the locus of indeterminacy is $\mathcal O \setminus \mathbb P(W^\perp )$;
\item[b] $\mathrm{Im}\, R_{gW}=\mathcal O (Q(W))$;
\item[c] if $x\in \mathcal O(Q(W))$, then $R_{gW}(x)=\rho(g) x$;
\item[d] $R_{gW}=L_{\rho(g)}\circ \hat{\pi_W }$, where $L_{\rho(g)}$ is the automorphism of $\mathcal O(Q(W))$ defined by $\rho(g)$;
\item[e] $\mathrm{Im}\, R_{gW}=\mathrm{Im}\, R_{g'W'}$ if and only if $W=W'$.
\end{enumerate}
\end{lemma}
\begin{proof}
The domain of $R_{gW}$ is the domain of $\hat{\pi}_W$ and so is given by $\mathcal O \setminus \mathbb P (W^\perp )$.

The Lie group $M(W)$ is compatible. By Proposition \ref{stabilizzatore-faccie},  it preserves $\mathcal O(Q(W))$ and so  $\rho(g)$ induces an automorphism of $\mathcal O (Q(W))$. By Theorem \ref{flow-parabolic}, $\hat{\pi}_{W} (\mathcal O)=\mathcal O(Q(W))$. Therefore $R_{gW}$ is given by the composition of the automorphism of $\mathcal O (Q(W))$ defined by $\rho(g)$ and the rational map $\hat{\pi}_W$. This proves item $[a],[b],[c],[d]$.
By Proposition \ref{orbita-sottospazio}, $\mathcal O(Q(W))$ determines $W$. Then  $\mathrm{Im}\, R_{gW}=\mathrm{Im}\, R_{g'W'}$ if and only if $W=W'$, concluding the proof.
\end{proof}
Given $W$ a $\tau$-connected subspaces, we have a map
\[
r_W: X_W \lra \{\mathrm{Rational\ maps\ of \ } \mathcal O\}, \qquad gK(W) \mapsto R_{gW},
\]
The following Lemma proves that the above map is injective.
\begin{lemma}\label{meoo}
Let $\tau:G \lra \mathbb P(V)$ be an irreducible representation. Let $\mathcal O$ be a closed orbit of $G$ and let $g\in G$. If $g$ fixed pointwise $\mathcal O$, then $g$ fixed pointwise any element of $\mathbb P(V)$.
\end{lemma}
\begin{proof}
Since the linear span of $\mathcal O$ coincides with $\mathbb P(V)$, the automorphism  $g$ fixed pointwise any element of $\mathbb P(V)$.
\end{proof}
Let
\[
{ri}_\tau: \XS \lra \{\mathrm{Rational\ maps\ of\ } \mathcal O\},
\]
denote the map defined as follow:
\[
(ri_\tau)_{|_{X_W}}=r_W.
\]
By Lemmata \ref{immagine-mappa-razionale} and \ref{meoo}, we have the following result.
\begin{thm}\label{Satake-razionale}
The map
\[
{ri}_\tau: \XS \lra \{\mathrm{Rational\ maps\ of\ } \mathcal O\}
\]
is injective.
\end{thm}
The next step is to understand convergence in the Satake compactification which kind of convergence we have of the corresponding rational maps.
\begin{lemma}\label{convergence}
Let $p_n \mapsto p\in \XS$ and let $\hat{p_n}$ and $\hat{p}$ be the corresponding rational map. Then $\hat{p_n} \mapsto \hat{p}$ uniformly on compact subsets of an open dense connected subset of $\mathcal O$.
\end{lemma}
\begin{proof}
By Lemma \ref{prop-PV}, we can find unique $A_n,A\in \mathcal H(V)$ such that $\mathrm{Tr}(A_n)=\mathrm{Tr}(A)=1$, $p_n=[A_n],p=[A] \in \XS$ and $A_n \mapsto A$ in  $\mathrm{End}(V)$. We denote by $W_n$, respectively $W$, the $\tau$-connected subspace of $V$ corresponding to $\hat{p_n}$, respectively $\hat{p}$. Then
\[
N=\bigcup_{n\in \mathbb N} \mathbb P(W_n^\perp ) \cup \mathbb P(W),
\]
has no interior point and so $A_n \mapsto A$ uniformly on the compact subset of $\mathcal O \setminus N$.  Hence, by Lemma \ref{sqrt-cont} we get the result.
\end{proof}
\section{The Bourguignon Li-Yau map}\label{srbly}
In this section we introduce the Bourguignon-Li-Yau for the $G$-action on $\mathbb P(V)$. This map has been studied by different levels of generality by Hersch \cite{hersch}, Bourguignon, Li and Yau \cite{bourguignon-li-yau},  Millson and Zombro  \cite{millson-zombro}, Biliotti and Ghigi \cite{biliotti-ghigi-American,bgs}.

If $M$ is a compact manifold, denote by $\mathscr M (M)$ the vector space of finite signed Borel measures on $M$.  These measures are
  automatically Radon \cite[Thm. 7.8, p. 217]{folland-real-analysis}.
  Denote by $C(M)$ the space of real continuous function on $M$. It is
  a Banach space with the $\mathrm{sup}$--norm.  By the Riesz Representation
  Theorem \cite[p.223]{folland-real-analysis} $\mathscr M (M)$ is the
  topological dual of $C(M)$.  The induced norm on $\mathscr M (M)$ is the
  following one:
  \begin{gather}
    ||\nu||:= \sup \Bigl \{ \int_M f d\nu: f\in C(M), \sup_M |f| \leq
    1 \Bigr\}.
  \end{gather}
  We endow $\mathscr M (M)$ with the weak-$*$ topology as dual of $C(M)$.
  Usually this is simply called the \emph{weak topology} on measures. We use
  the symbol $\nu_\alfa\rightharpoonup  \nu$ to denote the weak convergence of
  the net $\{\nu_\alfa \}$ to the measure $\nu$.  Denote by
  $\mathscr P (M) \subset \mathscr M (M)$ the set of Borel probability measures
  on $M$.  We claim $\mathscr P (M)$ is a compact convex subset of
  $\mathscr M (M)$. Indeed the cone of positive measures is closed and
  $\mathscr P (M)$ is the intersection of this cone with the closed affine
  hyperplane $\{\nu\in \mathscr M (M) : \nu(M) = 1\}$. Hence $\mathscr P (M)$ is
  closed.  For a positive measure $|\nu|=\nu$, so $\mathscr P (M)$ is
  contained in the closed unit ball in $\mathscr M (M)$, which is compact in
  the weak topology by the Banach-Alaoglu Theorem
  \cite[p. 425]{dunford-schwartz-1}.  Since $C(M)$ is separable, the
  weak topology on $\mathscr P (M)$ is metrizable \cite[p. 426]{dunford-schwartz-1}.

If $f : M \lra N$ is a measurable map between measurable spaces and
  $\nu $ is a measure on $M$, the \emph{image measure} $f_* \nu$ is
defined by  $f_* \nu(A) : = \nu (f^{-1} (A))$.  It
  satisfies the \emph{change of variables formula}
  \begin{gather}
    \label{eq:pushforward}
    \int_N u (y) \mathrm{d}( f_* \nu)(y) = \int_M u(f(x)) \mathrm{d}\nu(x).
  \end{gather}
In the sequel for $ g \in G$ and $ \nu \in \mathscr P(M)$, we will use the notation
\begin{gather*}
    g \cdot \nu := g_* \nu.
\end{gather*}
Let $\mathcal O$ be the unique closed orbit of $G$ on $\mathbb P(V)$ contained in the unique closed orbit of $U^\C$. Define
\[
\mathscr F: \mathscr P (\mathcal O) \lra \liep, \qquad \gamma \mapsto \int_{\mathcal O} \mup (x) \mathrm{d} \gamma (x).
\]
By a standard formula of change of variable, we have
\[
\mathscr F (g\cdot \gamma)=\int_{\mathcal O} \mup (x)\mathrm{d} ( g\cdot \gamma ) (x)=\int_{\mathcal O} \mup(gy)\mathrm{d}\gamma(y).
\]
Using the homomorphism
\[
X=G/K \lra G, \qquad gK \mapsto \sqrt{gg^*},
\]
we define the following map.
\begin{defin}\label{rbly}
Given a probability measure $\gamma$ of $\mathcal O$, the \emph{Bourguignon-Li-Yau map} $\Psi_\gamma:X\lra \liep$ is defined by
\[
\Psi_\gamma (gK)=\int_{\mathcal O} \mup(\rho(g)x)\mathrm{d} \mu(x),
\]
where $\rho(g)=\sqrt{gg^*}$ as in (\ref{root}).
\end{defin}
\begin{lemma}\label{basic}
$\Psi_\gamma (X) \subseteq \mathcal E$ and $0\in \mathrm{Int}(\mathcal E)$.
\end{lemma}
\begin{proof}
Since $\mathcal E=\mathrm{conv}(\mup(\mathcal O))$ and $\gamma$ is a probability measure, it follows $\Psi_\gamma (X)\subseteq \mathcal E$.

Let $\mathcal O'$ be the unique closed orbit of $U^\C$. Since $U^\C=U_0^\C\cdot U_1^\C$, $G=G_0 \cdot U_1^\C$ and $G_0$ is a real form of $U_0$,
then any simple factor of $U^\C$ acts non trivially on $V$. By Lemma 69 in \cite[p. $260$]{biliotti-ghigi-American}, we have $0\in \mathrm{Int}\left(\mathrm{conv} (i\mu(\mathcal O'))\right)$.

Let $\lia\subset \liep$ be a maximal Abelian subalgebra. By Proposition \ref{image-gradient},
$\mua(\mathcal O')=\mua(\mathcal O)$. Therefore
\[
\pi_\liep (\mathrm{conv} (i\mu(\mathcal O')))=K\mua(\mathcal O)=\mathcal E,
\]
where $\pi_\liep$ is the orthogonal projection of $i\mathfrak u$ onto $\liep$.  Since $\pi_\liep$ is an open map, it follows that $0\in \mathrm{Int}(\mathcal E)$.
\end{proof}
\begin{defin}\label{bly}
We say that a probability measure $\gamma$ on $\mathcal O$ is $\tau$-admissible if for any hyperplane $H \subset \mathbb P(V)$, then
$\gamma(\mathcal O \cap H)=0$
\end{defin}
\begin{lemma}\label{convergense2}
Assume that $p_n \mapsto p$ in $\XS$ and let $\hat{p_n}$ and $\hat{p}$ the corresponding rational maps. If $\gamma$ is $\tau$-admissible,
then $\hat{p_n} \mapsto \hat{p}$ $\gamma$-ae, i.e., $\gamma$ almost everywhere.
\end{lemma}
\begin{proof}
In Lemma \ref{convergence}, we prove that $\hat{p_n} \mapsto \hat{p}$ on a compact subset of $\mathbb P(V)\setminus N$,
where $N$ is a countably union of linear proper projective subspaces of $\mathbb P(V)$. Therefore $\gamma(N)=0$ and so $\hat{p_n} \mapsto \hat{p}$, $\gamma$-ae.
\end{proof}
\begin{lemma}\label{razionali-misure}
Let $\gamma$ be a $\tau$-admissible measure. Let $p\in X_W$ and let $\hat{p}$ the corresponding rational map. Then $\hat{p}_*\gamma$ is $\tau_W$-admissible measure. Moreover, if $\gamma$ is a $K$-invariant measure, then $\hat{p}_* \gamma$ is $K(W)$-invariant measure of $\mathcal O(W)$.
\end{lemma}
\begin{proof}
Let $H\subset W$ be an hyperplane. It is easy to check that $\hat{p}_* \gamma (H)=\gamma (\mathcal O \setminus (W^\perp \cup H) )$ and so it is $\tau_W$ admissible. The last statement follows from the fact that $\hat{\pi_W}$ is $K(W)$-equivariant.
\end{proof}
\begin{thm}
Let $\gamma\in \mathscr P (\mathcal O)$ be a $\tau$-admissible measure. Then the Bourguignon-Li-Yau map $\Psi_\gamma$ admits a continuous extension to $\XS$ that we still denote by $\Psi_\gamma$. This extension is unique and satisfies $\Psi_\gamma (\XS)\subseteq \mathcal E$.
\end{thm}
\begin{proof}
Let $p\in X_W$ and let $\hat{p}$ the corresponding rational map of $\mathcal O$. Then $\hat{p}=\rho(g)\hat{\pi_W}$, where $g\in M(W)$. The function $\mup \circ \hat{p}$ is defined on $\mathcal O \setminus\{W^\perp \}$ and so $\gamma$-ae. It is also bounded on $\mathcal O \setminus\{W\}$ and so it $\gamma$ integrable. Therefore
\[
\Psi_\gamma (p)=\int_{\mathcal O} \mup (\hat{p} (x)) \mathrm{d} \gamma (x)=\int_{\mathcal O \setminus \{W^\perp \}} \mup (\hat{p} (x)) \mathrm{d} \gamma (x),
\]
is well-posed.
If $W=V$, then $p=\rho(g)$ for some $g\in G$, and so the definition agree with the definition of the Bourguignon-Li-Yau given in Definition  \ref{bly}. By Lemmata \ref{convergence} and \ref{convergense2} the extension of $\Psi_\gamma$  is continuous. Finally, since $G/K$ is dense in $\XS$, the extension is unique.
\end{proof}
\begin{thm}\label{bordo}
Let $W\subset V$ be a $\tau$-connected subspace.
Then $\Psi_\gamma (X_W)\subseteq F_W$. Moreover,
$
(\Psi_\gamma)^{-1} (F_W)=\bigsqcup_{W'\subseteq W} X_W' =\overline{X_W}
$
\end{thm}
\begin{proof}
Let $p\in X_W$ and let $\hat{p}$ the corresponding rational map. By Lemma \ref{immagine-mappa-razionale}, we have  $\hat{p}(\mathcal O)=\mathcal O(Q(W))$. By Lemma \ref{razionali-misure}, $\hat{p}_* \gamma$ is $\tau_W$-admissible. Since $\mathcal O \setminus W^\perp $ has full measure, we have
\[
\Psi_\gamma (p)=\int_{\mathcal O} \mup (\hat{p} (x)) \mathrm{d} \gamma (x)=\int_{\mathcal O \setminus{W^\perp}} \mup (\hat{p} (x)) \mathrm{d} \gamma (x)=
\int_{\mathcal O(Q(W))} \mup(y) \mathrm{d} (\hat{p}_* \gamma)(y).
\]
Since $\hat{p}(\mathcal O)=\mathcal O(Q(W))$ and $\mup(\mathcal O(Q(W)))=\mathrm{ext} (F_W)$, it follows $\Psi_\gamma (X_W)\subset F_W$. This also proves
\[
(\Psi_\gamma)^{-1} (F_W)\supseteq\bigsqcup_{W'\subset W} X_W'.
\]
Let $p\in X_{W'}$ be such that $\Psi_\gamma (p) \in F_W$. Let $\xi \in \liep$ be such that $F_W=F_\xi (\mathcal E)$ and let
$\hat{p} =L_{\rho(g)}\circ \hat{\pi}_{W'}$ the corresponding rational map of $p$. Then
\[
\begin{split}
\mathrm{Max}_{z\in \mathcal E} \langle z, \xi \rangle &=
\langle \Psi_\gamma (p), \xi \rangle \\ &=\langle \int_{\mathcal O} \mup \left(\rho(g) \hat{\pi}_{W'} x \right) \mathrm{d} \gamma(x), \xi \rangle \\ &=
\int_{\mathcal O(Q(W'))} \langle \mup \left(\rho(g)y \right), \xi \rangle\, \mathrm{d} (\hat{\pi}_{W'})_* \gamma )(y)
\end{split}
\]
Since $\rho(g)$ is an automorphism of $\mathcal O(Q(W'))$, it follows that \[ \langle \mup(x),\xi \rangle  =\mathrm{Max}_{z\in \mathcal E} \langle z, \xi \rangle,
\]
$\hat{\pi_{W'}}_*\gamma$-ae. By Lemma \ref{razionali-misure}, $\hat{\pi_{W'}}_*\gamma$ is a $\tau_{W'}$-admissible.

Let  $q\in X_{W'}$. Then
\[
\begin{split}
\langle \Psi_\gamma (q), \xi \rangle &= \langle \int_{\mathcal O} \mup \left(\rho(g') \hat{\pi}_{W'} x \right) \mathrm{d} \gamma(x), \xi \rangle \\
&=\int_{\mathcal O(Q(W'))} \langle \mup \left(\rho(g')y \right), \xi \rangle\, \mathrm{d} (\hat{\pi}_{W'})_* \gamma)(y) \\ &=
\mathrm{Max}_{z\in \mathcal E} \langle z, \xi \rangle,
\end{split}
\]
and so $\Psi_\gamma (X_{W'})\subseteq F_W$. This implies $F_{W'}\subseteq F_W$. By Proposition \ref{tau-faccie}, we get $W'\subseteq W$ concluding the proof.
\end{proof}
\begin{cor}\label{bordo2}
Let $\gamma$ be a $\tau$-admissible measure. Then $\Psi_\gamma (\partial \XS)\subseteq \partial \mathcal E$.
\end{cor}
Our aim is to prove that for any $\tau$-admissible measure, $\Psi_\gamma (\XS)=\mathcal E$. We start, considering the smooth $K$-invariant measure  $\nu$ which is, of course, $\tau$-admissible.
\begin{thm}\label{bly-invariant-measure}
The Bourguignon-Li-Yau map $\Psi_\nu: \XS \lra \mathcal E$ is an homeomorphism. Moreover, for any $W\subseteq V$ $\tau$-connected subspace
\[
\Psi_\nu: X_W \lra \mathrm{relint} F_W,
\]
is an homeomorphism.
\end{thm}
\begin{proof}
By Corollary 5.3 $p.153$ in \cite{biliotti-raffero}, the map
\[
\mathscr F_\nu : G \lra \liep, \qquad g \mapsto \int_{\mathcal O} \mup(x) \mathrm{d} \nu(x),
\]
is a submersion onto $\mathrm{Int} (\mathcal E)$ which descents to a diffemorphism $\mathscr F_\nu: G/K \lra  \mathrm{Int} (\mathcal E)$. Since
\[
\Psi_\nu (gK)= \mathscr F_\nu (\rho(g)K),
\]
$\Psi_\nu$ is an homeomorphism.

Let $W\subset V$ be a $\tau$-connected subspace. By Lemma \ref{razionali-misure}, $(\hat{\pi_{W}})_* \nu$ is a $K(W)$ invariant measure on $\mathcal O(W)$ and so
\[
\Psi_\nu: X_W \lra \mathrm{relint}\, F_W,
\]
is an homeomorphism. By Theorems \ref{schneider-facce} and \ref{bordo}, we get that $\Psi_\nu:\XS \lra \mathcal E$ is an homeomorphism.
\end{proof}
\begin{thm}
  \label{grado}
  Let $\gamma$ be a $\tau$-admissible measure.  Then the
  Bourguignon-Li-Yau map $\Psi_\gamma : \XS\lra \mathcal E$ is surjective.
\end{thm}
\begin{proof}
Set $\gamma_t:=t\gamma + (1-t)\nu$. Define
\[
H: \XS \times [0,1] \lra \mathcal E \qquad H(p,t) := \Psi_{\gamma_t} (p)=t\Psi_\gamma (p) + (1-t) \nu_\gamma(p).
\]
$\gamma_t$ is $\tau$-admissible measure on $\mathcal O$ for every $t\in [0,1]$ and $H$ is
  continuous. By Theorem  \ref{bordo},  $\tau$-connected subspace $W\subseteq V$, we have $H( X_W \times [0,1] ) \subset
  F_W$ and so $H(\partial \XS \times [0,1] )\subset  \partial \mathcal E$. Since $H(\cdot,0)=\Psi_\nu (\cdot)$ is an homeomorphism, it has degree
  $1$. Hence the same holds for $H(\cdot, 1)=\Psi_\gamma$. By a classical
  topological argument this yields the surjectivity of $H(\cdot, 1) =\Psi_\gamma$.
\end{proof}
 Let $\lieg=\liek\oplus \liep$. There is a splitting of algebra, see \cite{helgason},
\[
\lieg=\lieg_1 \oplus \cdots \oplus \lieg_q,
\]
where $\lieg_i=k_i \oplus \liep_i$ is an ideal of $\lieg$, for $i=1,\ldots,q$,  and
\[
\liek=\liek_1 \oplus \cdots \oplus \liek_q, \qquad \liep=\liep_1 \oplus \cdots \oplus \liep_q.
\]
Let $G_1,\ldots,G_q$, $K_1,\ldots,K_q$ the corresponding analytic (connected) subgroups. Let $x=x_1+\cdots +x_p \in \liep$. Then
\[
K\cdot x= K_1 \cdot x_1 +\cdots + K\cdot x_p,
\]
and so one may easily check that
\[
\mathrm{con} (K\cdot x)=\mathrm{conv} (K_1 \cdot x_1)+\cdots +\mathrm{conv}(K_p \cdot x_p).
\]
Let $\tau:G\lra \mathrm{SL}(V)$ be an irreducible representation with finite kernel. Then any factor of $G$ acts non trivially on $V$.
By \cite[$p.197$]{goodman-wallach-Springer},
$\mathbb P(V)=\mathbb P(V_1 \otimes \cdots \otimes V_q)$, $\tau=\tau_1 \otimes \cdots \otimes \tau_q$
and the kernel of $\tau_i$ is finite for any $i=1,\ldots,q$. Then
\[
\mu_\tau =\mu_{\tau_1} + \cdots +\mu_{\tau_q},
\]
and $\mu_{\tau_i}$ is not zero for any $i=1,\ldots,q$. Therefore
\[
\mathcal E=\mathcal E_1 +\cdots + \mathcal E_q,
\]
where $\mathcal E_i$ is the orbitope associated to the projective representation $\tau_i :G_i \lra \mathrm{PSL}(V_i)$, for $i=1,\ldots p$.
By Theorem \ref{bly-invariant-measure}, we get the following result.
\begin{thm}\label{riducibile}
If $X$ is reducible, i.e., $X=X_1 \times \cdots \times X_q$, then
\[
\overline{X}_{\mu_\tau}^S=\overline{X_1}_{\mu_{\tau_1}}^S \times \cdots \times \overline{X_q}_{\mu_{\tau_q}}^S.
\]
\end{thm}
\section{Furstenberg compactifications}
\label{section-furst}
Let $G$ be a semisimple noncompact Lie group and $K$ a maximal compact subgroup.
Another way to compactly $X=G/K$ was found by Furstenberg
\cite{furstenberg-Poisson} in his search for an analogue of the
Poisson formula for the unit disc.  We recall very briefly the
definition. In the sequel, we always refer to \cite[\S I.6]{borel-ji-libro}).
\begin{defin}
A compact homogeneous space $\mathcal O$ is called \emph{boundary} of $G$ or a $G$-\emph{boundary} if for every probability measure $\mu$ on $\mathcal O$, there exists a sequence $g_j\in G$ such that $g_j \cdot \mu$ converges to the delta measures $\delta_x$ at some point of $x\in \mathcal O$. A $G$-boundary is called a boundary of $X$.
\end{defin}
Using Iwasawa structure theory, Moore \cite[Thm. 1]{moore} proved that $Y=G/P$ is a boundary
if and only if $P$ is parabolic.
Let
$\nu$ be the $K$-invariant measure on $\mathcal O$. Then the map
\begin{gather*}
G \lra \mathscr P (\mathcal O) \qquad g \mapsto g\cdot \nu
\end{gather*}
descends to a continuous map $i_{\mathcal O} : X=G/K \lra \mathscr F (\mathcal O)$, which is
injective if and only if $P$ does not contain simple factors of $G$ (see
\cite[Thm. 4]{moore} or
\cite[Prop. I.6.16]{borel-ji-libro}).  In this case $\mathcal O$ is called a
\enf{faithful} Furstenberg boundary and the set
\begin{gather}
  \label{eq:def-furstenberg}
  \overline{X}_{\mathcal O}^F: = \overline {i_{\mathcal O} (X)}
\end{gather}
is called the \emph{Furstenberg compactification} of $X$ associated to
the faithful boundary $M$.  Fix an irreducible complex
representation $\tau : G \lra \mathrm{GL} (V)$ such that $P$ is the stabilizer
of some $x_0\in \mathbb P (V)$ and $\ker \tau$ is finite. Such
representations always exist and so $\mathcal O$ can be identified with the unique closed orbit of $G$ in $\mathbb P (V)$ contained in the unique closed orbit of $U^\C$.
\begin{thm}
  \label{furst}
  The map
  \begin{gather*}
    \Gamma : \XS \lra \overline{X}_{\mathcal O}^F \qquad \Gamma(p) := \hat{p}_* \nu
  \end{gather*}
  is a $G$-equivariant homeomorphism of $\XS$ onto $\overline{X}_{\mathcal O}^F$ such that
  $i_M=\Gamma \circ i_\tau$ (compare \eqref{eq:imbeddoinpos}).
\end{thm}
\begin{proof}
Let $p\in \XS$ and $\hat p: \mathcal O \dashrightarrow \mathcal O $ be the corresponding rational map. $\hat p$
 is defined $\mu$-ae, since $\nu$ is smooth and so $\tau$-admissible. Therefore $\Ga(p) = \hat{p}_* \mu$ is
  well-defined for any $p\in \XS$.  By Lemma \ref{convergence}, if $p_ n \mapsto p$ in $\XS$,  then $\hat{p_n}
  \mapsto \hat{p}$ $\nu$-ae, and so
  $\Gamma(p_n) \rightharpoonup \Gamma (p)$. This proves that $\Gamma$ is continuous. Since $i_M (X)$ is dense in $\overline{X}_{\mathcal O}^F$ then the map is surjective. Now we prove it is injective.

Let $p,q\in \XS$ and let $\hat{p}$, $\hat{q}$ the corresponding rational map. If $\hat{p}_* \mu =\hat{q}_*\mu$, then
\[
\begin{array}{lcl}
\int_{\mathcal O} \mup(x) \mathrm{d} \hat{p}_* \mu (x)&=&\int_{\mathcal O} \mup(x) \mathrm{d} \hat{q}_* \mu (x) \\
\int_{\mathcal O} \mup(\hat{p}x)\mathrm{d} \nu (x)    &=&\int_{\mathcal O} \mup (\hat{q}x) \mathrm{d}\nu (x),
\end{array}
\]
and so
\[
\Psi_\nu (p)=\Psi_\nu (q),
\]
where $\Psi_\nu$ is the Bourguignon-Li-Yau with respect to the $K$-invariant metric $\nu$. Since $\Psi_\nu$ is an homeomorphism, $p=q$.
\end{proof}
\section{A remark on eigenvalue estimates}\label{laplacian}
Let $(M, \mathtt g)$ be a compact, connected orientable Riemannian manifold. It is well-known that the spectrum of the Laplacian $\Delta_g=-d^* d$,
acting on functions, form a discrete set. The first eigenvalue of the Laplacian operator, that we denote by $\lambda_1 (M,\mathtt g)$, is one of the most natural and studied Riemannian invariants. There has been a considerable amount of work devoted to estimating the first eigenvalue in terms of other geometric quantities associated to $(M,\mathtt g)$, see for instance \cite{apostolov-jakobson-kokarev,arezzo-ghigi-loi,berger,bgf,biliotti-ghigi-American,becker,kokarev,legrosa,panelli-podesta,yang-yau}. More precisely one would like to study the quantity
$
\lambda_1 (M,\mathtt g) \mathrm{Vol}(M,\mathtt g)^{n/2},
$
which is scale invariant. By the Rayleigh principle, upper bounds for the first eigenvalue are  obtained by constructing functions with zero mean, sensitive to the geometry of the underlying manifold. Indeed, if $f_1,\ldots,f_n \in \mathrm{C}^\infty (M)$ have zero mean with respect to $(M,\mathtt g)$, and so
\[
\int_M f_j (x) \mathrm{vol}_{\mathtt g} (x)=0,
\]
for $j=1,\ldots,n$, then
\begin{equation}\label{rayleight}
\lambda_1 (M,g)\leq \frac{\sum_{i=1}^n \int_M |\nabla f_j (x) |_{\mathtt g}^2 (x)\mathrm{vol}_{\mathtt g} (x)}{\sum_{i=1}^n \int_M f_j^2 (x) \mathrm{vol}_{\mathtt g} (x)} .
\end{equation}
In the paper \cite{hersch}, Hersch studied the first eigenvalue on the unite sphere $S^2\subset \R^3$.

Let $\mathtt g_0$ be the restriction on $S^2$ of the canonical scalar product of $\R^3$, normalized to have volume $4\pi$. Then $(S^2,\mathtt g_0)$ is an homogeneous K\"ahler manifold and its automorphism group is given by  $G=\operatorname{PSL}(2, \C)$ acting by M\"obius transformations. It is well-known that $\lambda_1 (S^2,\mathtt g_0)=2$ and   the three coordinate
functions $x,y,z$ are eigenfunctions  of the Laplacian. These functions are the components of
momentum map of $S^2$ with respect to $\mathrm{SO}(3)$-action. Hersch showed that if $\mathtt g$ is an
arbitrary Riemannian metric on $S^2$ (normalized to have volume $4\pi$), then there is $a\in G$ such that
$\int_{S^2} a^*x \mathrm{vol}_{\mathtt g} (p) = \int_{S^2} a^*y \mathrm{vol}_{\mathtt g}(p)=\int_{S^2} a^*z \mathrm{vol}_{\mathtt g} (p)=0$.
Moreover he showed that the right hand side in
\eqref{rayleight} is equal to $2$ and so
$\lambda_1(S^2, \mathtt g) \leq 2$.

Bourguignon, Li and Yau \cite{bourguignon-li-yau} realized that this method applies also to estimate $\lambda_1(\PP^n(\C), \mathtt g)$ if $\mathtt g$ is a
\emph{K\"ahler} metric.  In \cite{biliotti-ghigi-American} we recast
the method of Hersch-Bourguignon-Li-Yau in terms of momentum map and applied it when $M$ is an arbitrary Hermitian symmetric space,
$\mathtt g_0$ is the symmetric metric, $G=\mathrm{Aut(M)}$ and the functions
are the components of the momentum map $\mu: M \lra \mathfrak k$ for
$K:= \operatorname{Isom}(M, \mathtt g_0)$. The Bourguignon-Li-Yau map is the
tool to deal with the first step for the unique closed orbit of $G$ on $\OO'$. We use the gradient map instead the momentum map.
If $\mathtt g$ is a Riemannian metric on $\mathcal O$, we denote by $\nu:=\mathrm{vol}_{\mathtt g} / \mathrm{Vol}(\mathcal O,\mathtt g)$ the corresponding Borel probability measure. Let  $e_1 , \lds, e_r$ be an orthonormal basis of
$\liep$ and set $f_j := \langle \mup, e_j\rangle$.   By Theorem \ref{grado}, there exists
$a\in G$ such that $\int_{\mathcal O} \mup(a x) \mathrm{vol}_{\mathtt g} (x) = 0$ and so $\int_{\mathcal O} a^* f_j (x) \mathrm{vol}_{\mathtt g} (x)=0$ for  $j=1,\ldots, r$.  By Rayleigh's Theorem we get the following result.
\begin{thm}
In the above notation, we have
\begin{gather}
    \label{eq:stima}
    \lambda_1 (\mathcal O, \mathtt g) \leq \frac{\sum_{j=1}^r \int_M |\nabla (a^*f_j)|_{\mathtt g}^2
      \mathrm{vol}_{\mathtt g}}{\int_M a^*(|\mup|^2) \mathrm{vol}_{\mathtt g}}
\end{gather}
\end{thm}
The second step is to actually compute the right hand side in \eqref{rayleight}.  On the other hand, at the moment we are not able
to compute the right hand side in \eqref{eq:stima} except for the Hermitian symmetric spaces \cite{biliotti-ghigi-American, bgf}.  We believe that this computation can be carried out in much greater generality and that it would yield very interesting estimates. We leave this problem for future investigations.

\end{document}